\documentclass[10pt,reqno]{amsart}
\usepackage[english]{babel}
\usepackage{amsmath,amsthm,amscd}
\usepackage{amssymb}
\usepackage{enumerate}
\usepackage[T1]{fontenc}
\usepackage{palatino}
\usepackage{mathpazo}
\usepackage{paralist}
\usepackage[a4paper]{geometry}
\usepackage{url}
\usepackage{hyperref}
\usepackage[all]{xy}

\DeclareMathOperator{\id}{id}
\DeclareMathOperator{\im}{im}
\DeclareMathOperator{\cl}{cl}

\newcommand{\K}{\mathbb{K}}
\newcommand{\NN}{\mathbb{N}}
\newcommand{\QQ}{\mathbb{Q}}
\newcommand{\ZZ}{\mathbb{Z}}
\newcommand{\RR}{\mathbb{R}}
\newcommand{\CC}{\mathbb{C}}
\newcommand{\CP}{\mathbb{C}P}

\renewcommand{\bar}[1]{\overline{#1}}

\newcommand{\ev}{\mathrm{ev}}

\newcommand{\ind}{\mathrm{ind}}
\newcommand{\Fix}{\mathrm{Fix\;}}
\newcommand{\wgt}{\mathrm{wgt}}
\newcommand{\cwgt}{\mathrm{cwgt}}
\newcommand{\pr}{\mathrm{pr}}
\newcommand{\wgtTC}{\mathrm{wgt}_{\mathsf{TC}}}

\newcommand{\op}{\mathrm{op}}
\DeclareMathOperator{\Crit}{\mathrm{Crit\;}}

\newcommand{\D}{\mathcal{D}}
\newcommand{\Dop}{\mathcal{D}^{\mathrm{op}}}

\newcommand{\Dcl}{\mathcal{D}^{\mathrm{cl}}}
\newcommand{\C}{\mathcal{C}}

\newcommand{\MM}{\mathcal{M}}
\newcommand{\Acal}{\mathcal{A}}

\DeclareMathOperator{\TC}{\mathsf{TC}}
\DeclareMathOperator{\SC}{\mathsf{SC}}
\newcommand{\SCcl}{\mathsf{SC}^{\mathrm{cl}}}
\DeclareMathOperator{\cat}{\mathsf{cat}}
\DeclareMathOperator{\secat}{\mathsf{secat}}
\DeclareMathOperator{\ccr}{\mathsf{cr}}

\newcommand{\zero}{\mathbb{O}}

\theoremstyle{plain}
\newtheorem{theorem}{Theorem}[section]
\newtheorem{prop}[theorem]{Proposition}
\newtheorem{lemma}[theorem]{Lemma}
\newtheorem{cor}[theorem]{Corollary}  
\newtheorem{theo}{Theorem}
\newtheorem*{theo74}{Theorem 7.4}

\theoremstyle{definition}
\newtheorem{definition}[theorem]{Definition}
\newtheorem*{strategy}{Strategy for finding two closed geodesics}

\theoremstyle{remark}
\newtheorem{remark}[theorem]{Remark}

\numberwithin{equation}{section}

\begin{document}
\setlength{\parindent}{0cm}
 
 \title[Spherical complexities]{Spherical complexities with applications to closed geodesics}
 \author{Stephan Mescher}
 \address{Mathematisches Institut \\ Universit\"at Leipzig \\ Augustusplatz 10 \\ 04109 Leipzig \\ Germany}
\email{mescher@math.uni-leipzig.de}
 \date{\today}
 
  \begin{abstract}
We construct and discuss new numerical homotopy invariants of topological spaces that are suitable for the study of functions on loop and sphere spaces. These invariants resemble the Lusternik-Schnirelmann category and provide lower bounds for the numbers of critical orbits of $SO(n)$-invariant functions on spaces of $n$-spheres in a manifold. Lower bounds on these invariants are derived using weights of cohomology classes. As an application, we prove new existence results for closed geodesics on Finsler manifolds of positive flag curvature satisfying a pinching condition.
 \end{abstract}

 \maketitle

\setcounter{tocdepth}{1} 
\tableofcontents

\section*{Introduction}

\subsection*{Sectional categories}

The Lusternik-Schnirelmann category $\cat(X)$ of a topological space $X$ is a classical $\NN$-valued homotopy invariant. The main motivation for L. Lusternik and L. Schnirelmann to introduce this number was its relation to critical point theory. Given a Hilbert manifold $M$, the number $\cat(M)$ provides a lower bound on the number of critical points of a smooth function on $M$.  A comprehensive overview of the most important results about and applications of Lusternik-Schnirelmann category is given in the textbook of Cornea, Lupton, Oprea and Tanr\'e \cite{CLOT}.

In recent years, another $\NN$-valued homotopy invariant has caught topologists' attention. In \cite{FarberTC}, see also  \cite{FarberSurveyTC} and \cite{FarberBook}, M. Farber has defined the topological complexity $\TC(X)$ of a topological space $X$ which is motivated by the motion planning problem from robotics. Its topological properties haven't been fully explored yet. 

Both $\cat(X)$ and $\TC(X)$ are special cases of the notion of the sectional category (also called Schwarz genus) of a fibration which was originally introduced by A. Schwarz in \cite{SchwarzGenus}. The sectional category $\secat(p)$ of a fibration $p:E \to B$ is given as the minimal number of open domains required to cover $B$, such that $p$ admits a continuous local section over each of these open domains.

\subsection*{Critical points and spherical complexities}

Let $M$ be a closed manifold. In many geometric applications, e.g. the study of closed geodesics, one studies functions on the Hilbert manifold $\Lambda M := W^{1,2}(S^1,M)$ which is homotopy-equivalent to $LM=\C^0(S^1,M)$, the free loop space of $M$. One often studies functions that are invariant under the $G$-action on $\Lambda M$, where $G$ is a subgroup of $O(2)$, that is induced by the standard $O(2)$-action on the time parameter in $S^1$. Examples are given by energy functionals of Riemannian or Finsler metrics and by action functionals of autonomous Lagrangians. An interesting problem is to find the minimal number of $G$-orbits of critical points of such functionals and the number of $G$-orbits contained in certain sublevel sets of the function.

While this number cannot a priori be estimated by the Lusternik-Schnirelmann category, M. Clapp and D. Puppe have developed an equivariant Lusternik-Schnirelmann-type invariant in \cite{ClappPuppeSymm} and \cite{ClappPuppeGeod}, which provides lower bounds for numbers of critical orbits of invariant functions on manifolds with compact Lie group actions. In particular, Clapp and Puppe have obtained existence results for closed geodesics on simply connected Riemannian manifolds in this way. A similar approach was followed by T. Bartsch in the study of periodic orbits of certain Hamiltonian systems, see \cite{BartschLNM}.

In this article, we introduce numbers that are also suitable for estimating the numbers of orbits of  critical points of functions on loop and sphere spaces. Given a topological space $X$ and $n \in \NN$ we put $B_{n+1}X := \C^0(B^{n+1},X)$, with $B^{n+1}$ denoting the closed unit ball in $\RR^{n+1}$, and let $S_nX:= \{f \in \C^0(S^n,X) \; | \; \text{$f$ is nullhomotopic}\}$. We define the $n$-spherical complexity of $X$, denoted by $\SC_n(X)$, as the sectional category of the fibration
$$r_n: B_{n+1}X \to S_nX, \qquad \gamma \mapsto \gamma|_{S^n}.$$
For $n=0$, it is easy to see that $\SC_0(X) = \TC(X)$. In a similar way, given a subspace $A\subset S_nX$, we define $\SC_{n,X}(A)= \secat(i_A^*r_n)$, where $i_A: A \hookrightarrow S_n$ denotes the inclusion. Using the standard left $O(n+1)$-action on $B^{n+1}$, we can define $O(n+1)$-actions on $S_nX$ and $B_{n+1}X$ by putting $(A \cdot \gamma)(p) = \gamma(A^{-1}p)$. We consider $S_nX$ and $B_{n+1}X$ as equipped with these $O(n+1)$-actions as well as the restricted $G$-actions for any subgroup $G \subset O(n+1)$. The following theorem arises as a special case of Theorems \ref{TheoremnuF} and \ref{TheoremMainSC1} of this article.

\begin{theo}
\label{theo1}
Let $M$ be a closed Riemannian manifold, $n \in \NN$ and let $\MM\subset S_nM$ be a Riemannian Hilbert manifold. Let $F: \MM \to \RR$ be continuously differentiable and $G$-invariant, where $G$ is a closed subgroup of $O(n+1)$. If $F$ is constant on the space of constant maps and satisfies the Palais-Smale condition with respect to the metric on $\MM$, and if 
\begin{enumerate}[(i)]
\item $n=1$ and $\pi_1(M)$ is torsion-free, or 
\item every critical point of $F$ has trivial isotropy group with respect to the $G$-action on $\MM$,
\end{enumerate}
then for each $\lambda \in \RR$ the sublevel set $F^\lambda := F^{-1}((-\infty,\lambda])$ contains at least $\SC_{n,M}(F^\lambda)-1$ distinct $G$-orbits of non-constant critical points of $F$.
\end{theo}

The assumptions of the theorem are in particular met for the restrictions of energy functionals of Finsler metrics to $\Lambda M \cap S_1M$ if $\pi_1(M)$ is torsion-free. Note that in contrast to the Morse-theoretic approach to critical points of a function we do not demand any non-degeneracy condition on the critical points of $F$ in Theorem \ref{theo1}.

\subsection*{Lower bounds and cohomology}

Sectional categories are in general hard to compute explicitly, but lower bounds are often derived in terms of the cohomology ring of the respective base space. More precisely, it was shown by Schwarz in \cite{SchwarzGenus} that the sectional category of a fibration $p:E \to B$ is bounded from below by one plus the cup length of $\ker [p^*:H^*(B;R) \to H^*(E;R)]$ as an ideal in $H^*(B;R)$, where $R$ is any commutative ring.

Nevertheless, there are examples of fibrations for which the difference between this cup length and the actual value of $\secat(p)$ becomes arbitrarily big. An example is provided by $LS^2$, for which, as shown in \cite{FHloop}, the cup product of $H^*(LS^2;\RR)$ is trivial, while $\cat(LS^2)=+\infty$. 

Based on ideas of Fadell and Husseini for Lusternik-Schnirelmann category from \cite{FadellHusseini} that were extended by Y. Rudyak in \cite{RudyakWeight}, Farber and M. Grant introduced the notion of sectional category weight of a cohomology class with respect to a fibration $p:E \to B$ in \cite{FarberGrantSymm} and \cite{FarberGrantWeights}. The upshot is that given $u_1,\dots,u_n \in \ker p^*$ with $u_1 \cup u_2 \cup \dots \cup u_n \neq 0$, there are numbers $\wgt_p(u_i) \in \NN$, such that $\secat(p)$ will be bounded from below by $\sum_{i=1}^n \wgt_p(u_i)+1$. In particular, the mere existence of a class of weight at least $k$ implies that $\secat(p) \geq k+1$.

In \cite{TCsymp}, Grant and the author have studied the condition on a cohomology class to have weight two or bigger in the case of the fibration defining topological complexity. The property of a cohomology class of having $\wgt_p(u)\geq 2$ is equivalent to the property that $u \in \ker p_2^*$, where $p_2: E_2 \to B$ denotes the fiberwise join of $p$ with itself. Using the cohomological properties of fiberwise joins and de Rham theory for Fr\'echet manifolds, Grant and the author have computed the topological complexity of symplectically atoroidal manifolds.

In this article, we generalize their method to arbitrary fibrations with a particular focus on the fibrations $r_n: B_{n+1}X \to S_nX$. We transfer the approach from \cite{TCsymp} entirely to singular cohomology and establish similar geometric criteria on cohomology classes that allow us to construct classes $u \in H^*(S_nX;R)$ with $\wgt_{r_n}(u) \geq 2$. As a byproduct, we will derive several estimates on the topological complexity of closed manifolds.

In \cite{ClappPuppeGeod}, Clapp and Puppe showed their main results by estimating numbers of orbits of closed geodesics by lower bounds for the cup length in the $O(2)$-equivariant cohomology of free loop spaces of spheres and complex-projective spaces. We will not make use of equivariant cohomology anywhere in this article  and use only (non-equivariant) singular cohomology instead.

\subsection*{Closed geodesics on Finsler manifolds}

As an application of spherical complexities in the case $n=1$, we study closed geodesics of a Finsler metric of positive flag curvature on a closed manifold. Flag curvature of Finsler metrics is a notion generalizing the sectional curvature of a Riemannian manifold, see \cite{RadeNonrev}. While results on Riemannian manifolds of positive curvature have been obtained by W. Ballmann, G. Thorbergsson and W. Ziller in \cite{BTZclosed} and \cite{BTZExistence}. Several analogous results for Finsler metrics have been established as well. Most of them follow a Morse-theoretic approach and require the Finsler metric to be \emph{bumpy}, i.e. their energy functional has to be Morse-Bott and its critical points have to occur in isolated $S^1$-orbits.  This condition is fulfilled for generic choices of Finsler metrics, see the introduction of \cite{RadeSecond} for an overview and references.  

Another difficulty occuring in the study of energy functionals of Finsler metrics is their lack of regularity. In general, the energy functional of an arbitrary Finsler metric will be of class $C^{1,1}$, i.e. continuously differentiable with locally Lipschitz-continuous derivative, but not necessarily of class $C^2$. Hence, Morse-theoretic methods are not directly applicable to these functionals and require the introduction of additional techniques like the finite-dimensional approximation of sublevel sets. Our Lusternik-Schnirelmann-type approach to these functionals has the advantage that no such techniques are required.

There are already various results on closed geodesics of Finsler metrics that do not require any genericity condition on the metric. For example, V. Bangert and Y. Long have shown in \cite{BangertLong} that every Finsler metric on $S^2$ admits two geometrically distinct closed geodesics. The same has been proven for reversible Finsler metrics on $S^3$ by Long and H. Duan in \cite{LongDuan}. 

Recently, it has been shown by W. Wang in \cite{WangMultiple} that if the flag curvature of a Finsler metric $F$ on a closed manifold $M$ satisfies the pinching condition $\frac{\lambda^2}{(1+\lambda)^2} <K \leq 1$, where $\lambda$ denotes the reversibility of $F$, then $(M,F)$ admits $\lfloor \frac{\dim M +1}{2}\rfloor$ closed geodesics. As shown by Rademacher in \cite{RadeSphere}, this pinching condition implies that $M$ is homeomorphic to a sphere. 

Making use of our new approach using spherical complexities, we are able to prove an existence result for closed geodesics of positively curved Finsler metrics whose flag curvatures satisfy a more relaxed (by a factor of $\frac14$) pinching condition.

\begin{theo74}
Let $n \geq 3$ and let $M$ be a $2n$-dimensional closed oriented manifold. Assume that there exists a cohomology class $x \in H^{2k}(M;\QQ)$ with $x^2 \neq 0$, where $1 \leq k \leq \frac{n-1}2$. Let $F$ be a Finsler metric on $M$ of reversibility $\lambda$. If the flag curvature $K$ of $F$ satisfies 
$$\frac{1}{4}\Big(\frac{\lambda}{1+\lambda}\Big)^2 < \delta \leq K \leq 1,$$ 
then $(M,F)$ will have two positively distinct closed geodesics of length at most $\frac{\pi}{\sqrt{\delta}}$. 
If in addition $F$ is reversible, e.g. if $F$ is derived from a Riemannian metric whose sectional curvature satisfies $\frac{1}{16} <\delta \leq K \leq 1$, then $(M,F)$ will have two geometrically distinct closed geodesics of length at most $\frac{\pi}{\sqrt{\delta}}$.
\end{theo74}

Here, we call closed geodesics positively distinct if they do not lie in the same $S^1$-orbit, while we call them geometrically distinct if they do not lie in the same $O(2)$-orbit.

\subsection*{Structure of the article}

In Section \ref{SectionDefs} we present the construction of spherical complexities. We further establish some basic inequalities between spherical complexities and other invariants. 

In Section \ref{SectionLStheory} we study relative versions of spherical complexities and derive inequalities between spherical complexities and the numbers of critical orbits of $O(n+1)$-invariant functions on sphere spaces by a Lusternik-Schnirelmann type line of argument. 

Section \ref{SectionSC1} focuses on the case $n=1$, i.e. on free loop spaces. We extend results from Section \ref{SectionLStheory} in this case and carry out the details for action functionals of autonomous Lagrangians and energy functionals on loop spaces of Riemannian and Finsler metrics. 

In Section \ref{SectionWeights} we focus on lower bounds for spherical complexities given by the cohomology rings of sphere spaces. We employ the Mayer-Vietoris sequence of fiberwise joins to improve the lower bounds by cup lengths from Section \ref{SectionDefs}. We apply the notion of sectional category weight to spherical complexities and study cohomology classes of weight two or bigger in detail.

Section \ref{SectionGeomweight} extends the results of Section \ref{SectionWeights} by establishing more tangible and geometric criteria on certain cohomology classes of sphere spaces to have higher spherical complexity weights. The case of topological complexity is singled out and treated in detail.

In Section \ref{SectionTC} we apply the results of Sections \ref{SectionWeights} and \ref{SectionGeomweight} in the case $n=0$ to obtain estimates on the topological complexity of closed manifolds in terms of the existence of maps of non-zero degree between certain manifolds. 

Finally, in Section \ref{SectionGeod} we apply the previous results to energy functionals of Finsler metrics on closed manifolds. Together with the results of Section \ref{SectionSC1}, we derive an existence result for closed geodesics under additional assumptions on topology and curvature of the manifold.

\subsection*{Notation}
	
Given topological spaces $X$ and $Y$ we let $\C^0(X,Y)$ denote the space of continuous maps $X \to Y$. Given two manifolds $M$ and $N$ of class $C^k$ we further let $\C^k(M,N)$ denote the set of maps $M \to N$ of class $C^k$, where $k \in \NN\cup \{+\infty\}$. Throughout this article, $C^*$, $C_*$, $H_*$ and $H^*$ always denotes singular (co)chains and (co)homology, respectively.

\subsection*{Acknowledgements}

The author thanks Hans-Bert Rademacher for answering numerous questions on closed geodesics on Finsler manifolds, John Oprea, Alberto Abbondandolo and Felix Schlenk for helpful comments on earlier drafts of the manuscript and Mark Grant and Philip Kupper for valuable discussions. He further thanks the anonymous reviewer for a detailed and thoughtful review that the author found very helpful for improving the quality of the manuscript and Dieter Kotschick for pointing out a mistake in an earlier version of the article.

\section{Definitions and basic properties}
\label{SectionDefs}
Throughout this section, we let $X$ be a path-connected topological space. Given $n \in \NN_0$ we let $S^n$ be the unit sphere and $B^{n+1}$ be the closed unit ball in $\RR^{n+1}$ and define
$$S_nX :=\{ f \in \C^0(S^n,X) \ | \ f \text{ is nullhomotopic}\}, \qquad B_{n+1}X := \C^0(B^{n+1},X),$$
equipped with the compact-open topologies. Note that if $\pi_n(X)=\{1\}$, then $S_nX=\C^0(S^n,X)$.

\begin{definition}
Let $n \in \NN_0$.
\begin{enumerate}[(1)]
\item Let $A \subset S_nX$. A \emph{sphere filling over $A$} is a continuous map $s:A \to B_{n+1}X$ with $s(\gamma)|_{S^n}=\gamma$ for all $\gamma \in A$. 
\item $B\subset S_nX$ is called a \emph{sphere filling domain} if there exists a sphere filling over $B$. The set of all sphere filling domains in $S_nX$ will be denoted by $\D_n(X)$. We further put $$\Dop_n(X) := \{U \in \D_n(X) \ | \ U\subset S_nX \text{ is open} \}.$$
\item The \emph{$n$-spherical complexity of $X$} is given by
$$\SC_n(X) := \inf \Big\{k \in \NN \ \Big| \ \exists U_1,\dots,U_k \in \Dop_n(X) \ \text{with} \ \bigcup_{i=1}^k U_i = S_nX \Big\} \in \NN \cup \{+\infty\}. $$ 
\end{enumerate}
\end{definition}
\begin{remark}
We recall that the sectional category of a fibration $f:E \to B$ is denoted by $\secat(f) \in \NN\cup \{+\infty\}$ and is given as the minimal cardinality of an open covering of $B$ such that $f$ admits a continuous local section over each of the open sets in the covering. Using this notion, it holds that 
$$\SC_n(X)=\secat(r_n:B_{n+1} \to S_nX)$$
for each $n \in \NN$, where $r_n:B_{n+1}X \to S_nX$, $r_n(\gamma) = \gamma|_{S^n}$. Since the inclusion $i_n:S^n \hookrightarrow B^{n+1}$ is a cofibration, $r_n$ is a fibration by \cite[Theorem VII.6.13]{Bredon}.
\end{remark}

\begin{prop}
If $X$ and $Y$ are homotopy-equivalent topological spaces, then $\SC_n(X)=\SC_n(Y)$ for every $n \in \NN_0$. 
\end{prop}
\begin{proof}
Let $n \in \NN_0$, let $f:X \to Y$ and $g:Y \to X$ be continuous with $g \circ f \simeq \id_X$ and let $h:X \times [0,1] \to X$ be a homotopy from $g \circ f$ to $\id_X$. Put $r:= \SC_n(Y)$, let $V_1,\dots,V_r \in \Dop_n(Y)$ with $\bigcup_{i=1}^r V_i=S_nY$ and let $s_i:V_i \to B_{n+1}Y$ be a sphere filling for each $i$. We put $U_i := (S_nf)^{-1}(V_i)$, where $S_nf:S_nX \to S_nY$ is the map induced by $f$, and define
$$\sigma_i: U_i \to B_{n+1}X, \quad (\sigma_i(\gamma))(tx):= \begin{cases}
g(s_i(f \circ \gamma)(2tx)) & \text{if } t \in [0,\frac12], \\
h(\gamma(x),2t-1) & \text{if } t \in (\frac12,1],
\end{cases} \quad \forall t \in [0,1], \ x \in S^n.$$
for each $i \in \{1,2,\dots,r\}$. One checks that each $\sigma_i$ is a sphere filling, so $U_1,\dots,U_r \in \Dop_n(X)$. Moreover, $S_nX=\bigcup_{i=1}^rU_i$, which shows that $\SC_n(X) \leq r=\SC_n(Y)$. The opposite inequality follows analogously.
\end{proof}

For each $n \in \NN_0$ we let $c_n:X \to S_nX$ denote the inclusion of constant maps, i.e. $(c_n(x))(p) = x$ for all $x \in X$ and $p \in S^n$.

\begin{prop}
\label{Propcn}
The image $c_n(X)$ is a sphere filling domain for each $n \in \NN_0$.
\end{prop}
\begin{proof}
This is obvious, since the map $s: c_n(X) \to B_{n+1}X$, $(s(c_n(x)))(p) := x$ for all $x \in X$ and $p \in B^{n+1}$, is easily verified to be a sphere filling. 
\end{proof}

The next lemma gives an alternative characterization of sphere filling domains and generalizes a well-established lemma from the context of topological complexity, see \cite[Lemma 18.1]{FarberSurveyTC}. 

\begin{lemma}
\label{LemmaSCincl}
Assume that $X$ is a Hausdorff space and let $n \in \NN_0$ and $A \subset S_nX$. Then $A$ is a sphere filling domain if and only if the inclusion $A \hookrightarrow S_nX$ is homotopic to a map with values in $c_n(X)$. 
\end{lemma}
\begin{proof}
Assume there exists a continuous map $s: A \to B_{n+1}X$ with $s(\gamma)|_{S^n}=\gamma$ for all $\gamma \in A$. Let $\varphi: B^{n+1} \times [0,1] \to B^{n+1}$ be a strong deformation retraction of $B^{n+1}$ onto $\{1_n\}$, where $1_n=(1,0,\dots,0)\in \RR^{n+1}$, and define $H: A \times [0,1] \to S_nX$ by
$$(H(\gamma,t))(p) = (s(\gamma))(\varphi(p,t)) $$
for all $\gamma \in A$, $t \in [0,1]$ and $p \in S^n$. Since $X$ is Hausdorff, $H$ is continuous, see \cite[Theorem VII.2.10]{Bredon}. It further holds that
$$(H(\gamma,0))(p)=(s(\gamma))(p)=\gamma(p), \qquad (H(\gamma,1))(p)=(s(\gamma))(1_n)=\gamma(1_n)=(c_n(\gamma(1_n)))(p) $$
for all $p \in S^n$ and $\gamma \in A$. Hence, $H$ is a homotopy from the inclusion of $A$ to the map $A \to S_nX$, $\gamma \mapsto c_n(\gamma(1_n))$, which obviously takes values in $c_n(X)$ only.

Conversely, let $H:A \times [0,1] \to S_nX$ be continuous with $H(\gamma,0)=\gamma$ and $H(\gamma,1) \in c_n(X)$ for all $\gamma \in A$. Define a map $s: A \to B_{n+1}X$ by 
$$(s(\gamma))(rx) = (H(\gamma,1-r))(x) \quad \forall r \in [0,1], \ x \in S^n, \ \gamma \in A.$$
Since $H$ is continuous and $H(\gamma,1) \in c_n(X)$ for each $\gamma$, it follows that $s(\gamma):B^{n+1} \to X$ is indeed continuous for every $\gamma \in A$, so that $s$ is well-defined. 
Moreover, $(s(\gamma))(x)=(H(\gamma,0))(x)=\gamma(x)$ for all $x\in S^n$ and $\gamma \in A$, hence $s(\gamma)|_{S^n} =\gamma$ for all $\gamma \in A$, which shows the claim.
\end{proof}

The previous lemma can be employed to derive a first lower bound on $\SC_n(X)$. We will make use of the following lemma, which is an analogue to a lemma of R. Fox, see \cite[(14.2)]{Fox} or \cite[Lemma 1.11]{CLOT}, from the context of Lusternik-Schnirelmann category. We will mimic Fox's proof.

\begin{lemma}
\label{LemmaFox}
Let $A \subset S_nX$ be a closed sphere filling domain. If $X$ is a metrizable absolute neighborhood retract (ANR), then $A$ has an open neighborhood $U \subset S_nX$ that is a sphere filling domain. 
\end{lemma}
\begin{proof}
We first observe that if $X$ is a metrizable ANR, then $S_nX$ will be an ANR as well by \cite[Theorem VI.2.4]{HuRetr}. By the proof Lemma \ref{LemmaSCincl}, there exists a homotopy $H':A \times [0,1] \to S_nX$ from the inclusion $A \hookrightarrow S_nX$ to $\gamma \mapsto c_n(\gamma(1))$. Put $Q := (S_nX \times \{0\}) \cup (A \times [0,1]) \cup (S_nX \times \{1\})$ and define $H: Q \times [0,1]\to S_nX$ by 
$$H(\gamma,0)=\gamma, \ H(\gamma,1)= c_n(\gamma(1)) \ \forall \gamma \in X, \ \ H(a,t)=H'(a,t) \ \forall a \in A, t \in [0,1].$$
One checks that $H$ is well-defined and continuous. Since $S_nX$ is an ANR, there exists an open neighborhood $V \subset S_nX \times [0,1]$ of $Q$ and a continuous extension $K:V \to S_nX$ of $H$. Since $[0,1]$ is compact, $V$ contains a subset of the form $U \times [0,1]$ where $U$ is an open neighborhood of $A$ in $S_nX$. Then, by definition, $K|_{U \times [0,1]}$ is a homotopy from the inclusion $U \hookrightarrow S_nX$ to a map with values in $c_n(X)$. Thus, $U$ is a sphere filling domain by Lemma \ref{LemmaSCincl}, which applies since $X$ is metrizable, hence Hausdorff.
\end{proof}

The following statement generalizes a result of Farber for topological complexity, see \cite[Lemma 18.3]{FarberSurveyTC}. Its proof mimics Farber's proof in a straightforward way.

\begin{prop}
\label{PropSCquotcat}
If $X$ is a metrizable ANR, then
$$\SC_n(X) \geq \cat(S_nX/c_n(X))-1 \qquad \forall n \in \NN_0.$$
\end{prop}
\begin{proof}
Fix $n \in \NN$ and let $q: S_nX \to S_nX/c_n(X)$ denote the quotient space projection. Let $r:= \SC_n(X)$ and let $U_1,\dots,U_r \in \Dop_n(X)$ with $S_nX= \bigcup_{i=1}^r U_i$. By Lemma \ref{LemmaSCincl},the inclusion $U_i \hookrightarrow S_nX$ is for each $i$ homotopic to a map with values in $c_n(X)$. For each $i$ we put $$U_i':= p(U_i \setminus c_n(X)) \subset S_nX/c_n(X).$$ Then the inclusions $U_i'\hookrightarrow S_nX/c_n(X)$ will be nullhomotopic and by construction 
$$\bigcup_{i=1}^k U_i' =(S_nX/c_n(X)) \setminus \{x_0\},$$ 
where $x_0 \in S_nX/c_n(X)$ is given by $\{x_0\}=q(c_n(X))$. By Lemmas \ref{LemmaSCincl} and \ref{LemmaFox}, $c_n(X)$ has an open neighborhood $U_0 \subset S_nX$, such that $U_0 \hookrightarrow S_nX$ is homotopic to a map with values in $c_n(X)$. Thus, $U_0' := q(U_0)$ is an open neighborhood of $x_0$ that contracts onto $x_0$. We have shown that $\{U_0',U_1',\dots,U_k'\}$ is an open cover of $S_nX/c_n(X)$ for which the inclusion of each $U_i'$ is nullhomotopic, which implies $\cat (S_nX/c_n(X)) \leq \SC_n(X) +1$.
\end{proof}

We next consider analogues of spherical complexities for basepoint-preserving maps $S^n \to X$.

\begin{definition}
For each $n \in \NN_0$ we let $1_n=(1,0,\dots,0) \in S^n$. Let $x_0 \in X$ be a fixed basepoint and consider the subspaces 
$$S_n^*X := \{ f \in S_nX \ | \ f(1_n)=x_0\}, \qquad B_{n+1}^*(X) := \{ f \in B_{n+1}X \ | \ f(1_n)=x_0\}=r_n^{-1}(S_n^*(X)).$$
The \emph{based $n$-spherical complexity of $X$} is defined as the number
$$\SC^*_n(X) := \secat(r_n|_{B_{n+1}^*X}: B_{n+1}^*X \to S_n^*X ). $$ 
Note that this number is well-defined, since $r_n|_{B_{n+1}^*X}$ is a fibration by \cite[Corollary VII.6.16]{Bredon}.
\end{definition}

The numbers $\SC_0(X)$ and $\SC_0^*(X)$ are recognized as well-known invariants.

\begin{prop}
\label{PropSC0}
We have
$$\SC_0(X)= \TC(X) \qquad \text{and} \qquad \SC^*_0(X)=\cat(X), $$
where $\TC$ denotes topological complexity and $\cat$ denotes Lusternik-Schnirelmann category.
\end{prop}
\begin{proof}
Fix $x_0 \in X$. In terms of sectional category, $\TC$ and $\cat$ are given as
$$\TC(X) = \secat(\pi:PX \to X \times X), \qquad \cat(X) = \secat(\pi_0:P_{x_0}X \to X), $$
where $PX=\C^0([0,1],X)$, $\pi(\gamma)=(\gamma(0),\gamma(1))$, $P_{x_0}X=\{\gamma \in PX \ | \ \gamma(0)=x_0\}$, $\pi_0(\gamma)=\gamma(1)$. But obviously, there are homeomorphisms $B_1 X \approx PX $, $S_0X=\C^0(S^0,X)\approx X \times X$, $B_1^*X\approx P_{x_0}X$, $S_0^*X= \C^0(\{1\},X)\approx X$, which are compatible with the fibrations $r_0$, $r_0|_{B_1X}$, $\pi$ and $\pi_0$, which shows the claim.
\end{proof}

While there is no general identification of spherical complexities for $n>0$ as topological complexities or Lusternik-Schnirelmann (LS) categories, \emph{based} spherical complexities are indeed given as LS categories under mild assumptions.

\begin{lemma}
\label{LemmaBasedCat}
If $X$ is a Hausdorff space, then $\SC_n^*(X) = \cat(S_n^*X)$ for every $n \in \NN_0$.
\end{lemma}
\begin{proof}
Put $k := \cat(S_n^*X)$ and let $\{U_1,\dots,U_k\}$ be an open cover of $S_n^*X$ such that each inclusion $U_j \hookrightarrow S_n^*X$ is nullhomotopic. Then the inclusion is, in particular, homotopic to a map with values in $c_n(X)$, so by Lemma \ref{LemmaSCincl}, there exists a sphere filling $s_j:U_j \to B_{n+1}X$ for every $j \in \{1,2,\dots,k\}$ and clearly $s_j(U_j) \subset B^*_{n+1}X$. Since $\{U_1,\dots,U_k\}$ is a cover of  $S_n^*X$, this yields $\SC_n^*(X) \leq k$. 

Conversely, let $\ell := \SC_n^*(X)$ and let $\{V_1,\dots,V_\ell\}$ be a cover of $S_n^*X$ by sphere filling domains. Consider the homotopy $H$ from the proof of Lemma \ref{LemmaSCincl} with respect to $A=V_j$ for some $j \in \{1,2,\dots,\ell\}$. One checks that by construction, this homotopy takes values in $S_n^*X$ only since $V_j \subset S_n^*X$, so we obtain a homotopy $H_0:V_j \times [0,1] \to S_n^*X$ from the inclusion $V_j \hookrightarrow S_n^*X$ to the constant map $V_j \to S_n^*X$,  $\gamma \mapsto  c_n(\gamma(1))=c_n(x_0)$. Thus, $V_j \hookrightarrow S_n^*X$ is nullhomotopic for each $j \in \{1,2,\dots,\ell\}$. Consequently, $\cat(S_n^*X) \leq \ell = \SC_n^*(X)$.
\end{proof}

In particular, if $\pi_n(X)=0$, then 
$$\SC_n^*(X) = \cat(\Omega^nX),$$
where $\Omega^nX=\C^0((S^n,1),(X,x_0))$ denotes the $n$-fold iterated based loop space of $X$. We next derive estimates for $\SC_n(X)$ in terms of LS categories.

\begin{theorem}
\label{TheoremSCcatbound}
Assume that $X$ is a Hausdorff space. For all $n \in \NN$ it holds that
$$\cat(S_n^*X) \leq \SC_n(X) \leq \cat(S_nX).$$
\end{theorem}
\begin{proof}
The right-hand inequality is an immediate consequence of \cite[Theorem 18]{SchwarzGenus}. Moreover, since $B_{n+1}^*X=r_n^{-1}(S_n^*X)$, it follows from \cite[Proposition 7]{SchwarzGenus} that $\SC_n^*(X) \leq \SC_n(X)$, so the left-hand inequality is a consequence of Lemma \ref{LemmaBasedCat}.
\end{proof}

Note that by Proposition \ref{PropSC0}, the inequalities from Theorem \ref{TheoremSCcatbound} coincide for $n=0$ with Farber's well-established inequalities,  see \cite[Theorem 5]{FarberTC},
$$\cat(X) \leq \TC(X) \leq \cat(X \times X).$$
For $n\geq 1$, the inequalities from Theorem \ref{TheoremSCcatbound} imply that $\SC_n(X)$ will be infinite for many spaces appearing in practical applications. It was shown by F. R. Cohen in \cite{FRCohenLS} that if $X$ is simply connected and has the homotopy type of a finite CW complex, then $\cat(S_n^*X)=+\infty$ for all $n \geq 1$. Thus, Theorem \ref{TheoremSCcatbound} has the following consequence.

\begin{cor}
If $X$ is simply connected and has the homotopy type of a finite CW complex, then $\SC_n(X)=+\infty$ for all $n \geq 1$. 
\end{cor}


In \cite{SchwarzGenus}, Schwarz has established lower bounds on sectional categories in terms of the cohomology rings of their base spaces, which we want to make precise for spherical complexities. Given a topological space $X$, a commutative ring $R$ and an ideal $I$ of the cohomology ring $H^*(X;R)$, we let 
$$\cl_R(I) := \sup \{r \in \NN \ | \ \exists u_1,\dots,u_r \in I\ \text{with} \ u_1 \cup \dots \cup u_r\neq 0 \ \text{and} \ \deg u_j >0 \ \forall j \in \{1,2,\dots,r\}\}.$$
Given a subspace $A \subset X$, we further put
$$\cl_R(X,A) := \cl_R(H^*(X,A;R)), \qquad \cl_R(X) := \cl_R(H^*(X;R)).$$
\begin{theorem}
\label{TheoremSCcup}
Assume that $X$ is a Hausdorff space and let $R$ be a commutative ring. Then
$$\SC_n(X) \geq \cl_R (S_{n}X,c_n(X))+1\qquad \forall n \in \NN_0.$$
\end{theorem}
\begin{proof}
Fix $n \in \NN_0$. By \cite[Theorem 4]{SchwarzGenus}, it holds that 
$$\SC_{n}(X) \geq \cl_R\left(\ker \left[r_n^*:H^*(S_nX;R) \to H^*(B_{n+1}X;R)\right] \right)+1.$$
Let $\varphi_{n}: B^{n+1} \times [0,1] \to B^{n+1}$ be a strong deformation retraction of $B^{n+1}$ onto $\{1\}$ and put
$$ h_n:B_{n+1}X \times [0,1] \to S_nX, \qquad (h_n(\gamma,t))(p)= \gamma(\varphi_{n}(p,t)) \quad \forall \gamma \in B_{n+1}X, \  t \in [0,1], \ p \in S^n.$$
Since $X$ is Hausdorff, $h_n$ is continuous, so $h_n$ is a homotopy from $r_n$ to $c_n \circ e_n$, where we put $e_n:B_{n+1}X \to X$, $\gamma \mapsto \gamma(1)$. This $e_n$ is dual to the inclusion $i_n: \{1\} \hookrightarrow B^n$, which is a homotopy equivalence and hence so is $e_n$ by \cite[Theorem VII.2.8]{Bredon}. Thus, in cohomology, $\ker r_n^*= \ker (e_n^* \circ c_n^*) = \ker c_n^*$, which implies $\SC_n(X) \geq \cl_R\left(\ker c_n^*  \right)+1.$
Since $c_n$ has the left-inverse $\ev_n:S_nX \to X$, $\gamma \mapsto \gamma(1)$, the long exact cohomology sequence of the pair $(S_nX,c_n(X))$ breaks into short exact sequences
$$0 \to H^k(S_nX,c_n(X);R) \to H^k(S_nX;R) \stackrel{c_n^*}\to H^k(X;R) \to 0$$
for all $k \in \NN$. Thus, $\ker c_n^* = \im [ H^*(S_nX,c_n(X);R) \to H^{*}(S_nX;R)]$
and since the latter map is injective, 
$$\cl_R (\im [ H^k(S_nX,c_n(X);R) \to H^{k}(S_nX;A)]) = \cl_R(S_nX,c_n(X)),$$
which shows the claim. 
\end{proof}

\begin{cor}
Let $n \in \NN_0$ and $R$ be a commutative ring. If $X$ is Hausdorff and $\pi_n(X)=0$, then 
$$\SC_n(X) \geq \cl_R(\C^0(S^n,X),c_n(X))+1.$$
\end{cor}

\section{Lusternik-Schnirelmann theory for spherical complexities}
\label{SectionLStheory}
In this section, we want to apply methods from Lusternik-Schnirelmann theory, which in its classical formulation uses deformation results for sublevel sets of differentiable functions to obtain upper estimates on relative LS categories of sublevel sets by the numbers of critical points lying in them. Similar arguments will apply to spherical complexities, though we will see that in contrast to LS category, spherical complexities are lower bounds for the \emph{number of $O(n)$-orbits} of critical points of suitable $O(n)$-invariant functions. \\

Throughout this section, we again let $X$ be a path-connected topological space. The discussion will be similar to and oriented on the one in Section 1.2 and 1.7 of \cite{CLOT} for LS categories.

We begin by introducing subspace complexities that are defined in analogy with subspace LS categories, see \cite[Definition 1.1]{CLOT}.

\begin{definition}
Let $X$ be a topological space and $n \in \NN$. For $A \subset S_nX$ non-empty we put $B_{n+1}(X,A) := r_n^{-1}(A)$ and $r_{n,A} := r_n|_{B_{n+1}(X,A)}$ and define 
$$\SC_{n,X}(A) := \secat\left(r_{n,A}:B_{n+1}(X,A) \to A \right).$$
We further put $\SC_{n,X}(\varnothing):=0$.
\end{definition}

For these numbers we establish a relative version of the lower bound from Theorem \ref{TheoremSCcup}.

\begin{prop}
\label{PropSCrelCup}
Let $R$ be a commutative ring, $n \in \NN$ and $A \subset S_nX$. Then 
$$\SC_{n,X}(A) \geq \cl_R\left( \im\left[\iota_A^*:H^*(S_nX,c_n(X);R) \to H^*(A;R)\right] \right) +1,$$
where $\iota_A:(A,\varnothing) \hookrightarrow (S_nX,c_n(X))$ is the inclusion of pairs.
\end{prop}
\begin{proof}
By Schwarz's cup length lower bound for sectional categories, see \cite[Theorem 4]{SchwarzGenus}, it holds that
$$\SC_{n,X}(A) \geq \cl_R \left(\ker \left[r_{n,A}^*:H^*(A;R) \to H^*(B_{n+1}(X,A);R) \right] \right)+1.$$
The following diagram obviously commutes, where the horizontal maps are the respective inclusions:
$$\begin{CD}
B_{n+1}(X,A) @>{j_0}>>B_{n+1}X \\
@VV{r_{n,A}}V @VV{r_n}V \\
A @>{j_1}>> S_nX.
\end{CD} $$
Thus, if $u \in H^*(S_nX;R)$ lies in $\ker r_n^*$, then $j_1^*u\in \ker r_{n,A}^*$, hence 
$\SC_{n,X}(A) \geq \cl_R( j_1^*(\ker r_n^*))+1$. We derive from the long exact sequence of $(S_nX,c_n(X))$ that $\ker r_n^* = \im f_n^*$, where $f_n:(S_nX,\varnothing) \to (S_nX,c_n(X))$ denotes the inclusion of pairs. Hence, $$\SC_{n,X}(A) \geq \cl_R( \im ( f_n \circ \tilde{j}_1)^*)+1,$$ with $\tilde{j_1}:(A,\varnothing) \to (S_nX,\emptyset)$ denoting the inclusion of pairs. This shows the claim, since evidently $f_n \circ \tilde{j}_1=\iota_A$.
\end{proof}

The following statement is shown along the lines of Proposition \ref{PropSCquotcat}. We thus omit its proof.

\begin{prop}
Let $X$ be a metrizable ANR, $n \in \NN_0$ and $A\subset S_nX$ with $c_n(X) \subset A$. Then
$$\SC_{n,X}(A) \geq \cat_{S_nX/c_n(X)}(A/c_n(X))-1.$$
\end{prop}

The following proposition bundles various properties that will be required to carry out a Lusternik-Schnirelmann-type theory of spherical complexities. 

\begin{prop}
\label{PropSCprops}
\begin{enumerate}[(i)]
\item \emph{(Monotonicity)} \enskip If $A\subset B \subset S_nX$, then $$\SC_{n,X}(A) \leq \SC_{n,X}(B).$$
\item \emph{(Subadditivity)} For $A,B \subset S_n X$ it holds that $$\SC_{n,X}(A\cup B)\leq \SC_{n,X}(A) + \SC_{n,X}(B).$$
\item If $S_nX$ is normal and $A,B \subset S_nX$ satisfy $\bar{A} \cap \bar{B} = \varnothing$, then $$\SC_{n,X}(A \cup B) = \max\{ \SC_{n,X}(A),\SC_{n,X}(B)\}.$$
\end{enumerate}
\end{prop}
\begin{proof}
Statements (i) and (ii) are obvious. For statement (iii) let $U_0,V_0 \subset S_nX$ be open with $\bar{A} \subset U_0$, $\bar{B} \subset V_0$ and $U_0 \cap V_0 = \varnothing$. Put $r:= \SC_{n,X}(A)$, $s := \SC_{n,X}(B)$ and assume w.l.o.g. that $r \geq s$. Let $U_1,\dots,U_r,V_1,\dots V_s \in \Dop_n(X)$ with $A \subset \bigcup_{i=1}^rU_i$ and $B \subset \bigcup_{i=1}^sV_i$. If we put $U'_i := U_i \cap U_0$ for $1 \leq i \leq r$ and $V'_j := V_j \cap V_0$ for $1 \leq j \leq s$, then 
$$\{U'_1\cup V'_1,\dots,U'_s \cup V'_s, U'_{s+1},\dots,U'_r\} \subset \Dop_n(X)$$
will be an open cover of $A \cup B$ by sphere filling domains. Hence, 
$$\SC_{n,X}(A \cup B) \leq r= \max\{\SC_{n,X}(A),\SC_{n,X}(B)\}.$$ 
The converse inequality follows directly from (i).
\end{proof}

Another crucial property of invariants admitting a Lusternik-Schnirelmann-type theory is their behavior with respect to deformations. For this purpose we next introduce "closed versions" of subspace complexities and later relate them to the original ones.

\begin{definition}
For $n \in \NN_0$ let $\Dcl_n(X) := \left\{C \in \D_n(X) \ | \ C \text{ is closed in }S_nX\right\}$. For $A \subset S_nX$ we put 
$$\SCcl_{n,X}(A):=\inf \Big\{k \in \NN \ \Big| \ \exists C_1,\dots,C_k \in \Dcl_n(X) \ \text{with} \ A \subset \bigcup_{i=1}^k C_i \Big\} . $$
\end{definition}

Obviously, $\SCcl_{n,X}(\cdot)$ has the analogous properties of parts (i) and (ii) of Proposition \ref{PropSCprops}.

\begin{lemma}
\label{LemmaSCconst}
For each $n \in \NN_0$, it holds that $\SCcl_{n,X}(c_n(X))=1$.
\end{lemma}
\begin{proof}
This follows immediately from Proposition \ref{Propcn} since $c_n(X)$ is a closed subset of $S_nX$. 
\end{proof}

\begin{prop}[Deformation monotonicity]
\label{PropDeform}
If $A \subset S_nX$ is closed and $\Phi:A \times [0,1] \to S_nX$ is a deformation of $A$, then 
$$\SCcl_{n,X}(A) \leq \SCcl_{n,X}(\Phi_1(A)).$$
\end{prop}
\begin{proof}
Let $k := \SCcl_{n,X}(\Phi_1(A))$ and $A_1,A_2,\dots,A_k \in \Dcl_n(X)$ with $$\Phi_1(A)\subset\bigcup_{j=1}^k A_j.$$ For each $j \in \{1,2,\dots,k\}$ let $s_j:A_j \to B_{n+1}X$ be a sphere filling and put $B_j := \Phi_1^{-1}(A_j)$. Since $\Phi_1$ is continuous, $B_j$ is closed in $A$ for each $j$ and since $A$ itself is closed in $S_nX$, $B_j$ is a closed subset of $S_nX$ for every $j \in \{1,2,\dots,k\}$.
Define $\widetilde{s}_j: B_j \to B_{n+1}X$ for each $j$ by 
$$\widetilde{s}_j(\gamma):B^{n+1} \to X, \quad \left(\widetilde{s}_j(\gamma)\right)(r \cdot x) = \begin{cases}
(s_j(\Phi(\gamma,1)))(2rx) & \text{if } r \in [0,\frac12], \\
(\Phi(\gamma,2(1-r)))(x) & \text{if } r \in (\frac12,1],
\end{cases}$$
for all $x \in S^n$, $r \in [0,1]$ and $\gamma \in B_j$. One checks without difficulties that any such $\widetilde{s}_j$ is indeed a sphere filling. Since $A \subset \bigcup_{j=1}^k B_j$, this shows that $\SCcl_{n,X}(A)\leq k$.
\end{proof}

\begin{prop}[Continuity]
\label{PropCont}
If $X$ is a metrizable ANR and $A \subset S_nX$ is closed, then there will be an open neighborhood $U$ of $A$ with 
$\SCcl_{n,X}(A)=\SCcl_{n,X}(U)=\SCcl_{n,X}(\bar{U}).$
\end{prop}
\begin{proof}
As in the proof of Lemma \ref{LemmaFox}, $S_nX$ is a metrizable ANR since $X$ has that property. In particular, $S_nX$ is normal. Now let $r:= \SCcl_{n,X}(A)$ and let $A_1,\dots,A_r \in \Dcl_n(X)$ with $A \subset \bigcup_{i=1}^r A_i$. By Lemma \ref{LemmaFox}, there are $U_1,\dots,U_r \in \Dop_n(X)$ with $A_i \subset U_i$ for all $i \in \{1,2,\dots,r\}$. Since $S_nX$ is normal, we can separate $A_i$ and $X \setminus U_i$ by open subsets and obtain an open neighborhood $V_i$ of $A_i$ with $\bar{V}_i \subset U_i$ for every $i \in \{1,2,\dots,r\}$. In particular, $\bar{V}_i \in \Dcl_n(X)$ for each $i$. Put $U := \bigcup_{i=1}^r V_i$. Then $U$ is an open neighborhood of $A$ with $\bar{U} \subset \bigcup_{i=1}^r \bar{V}_i$, such that $$\SCcl_{n,X}(U) \leq \SCcl_{n,X}(\bar{U}) \leq r = \SCcl_{n,X}(A).$$ 
The converse inequalities are obvious, since $A \subset U \subset \bar{U}$.
\end{proof}

In the following proposition, we are again mimicking the proof of the corresponding result for LS categories, see \cite[Proposition 1.10]{CLOT}.
\begin{prop}
\label{PropSCclop}
If $X$ is a metrizable ANR, then $\SC_{n,X}(A) = \SCcl_{n,X}(A)$ for all closed $A \subset S_nX$.
\end{prop}
\begin{proof}
It is a direct consequence of Lemma \ref{LemmaFox} that $\SC_{n,X}(A) \leq \SCcl_{n,X}(A)$ whenever $X$ is a metrizable ANR. To show the opposite inequality, we let $k=\SC_{n,X}(A)$ and  $\{U_1,\dots,U_k\}\subset \Dop_n(X)$ be a cover of $A$ and put $U_j^c=S_nX\setminus U_j$ for every $j \in \{1,2,\dots,k\}$. Since $S_nX$ is metrizable, it is a normal space. In particular, the closed set $\bigcap_{j=1}^k U_j^c$ has an open neighborhood $W$ with $W \cap A=\varnothing$. Then $\{U_1,\dots,U_k,W\}$ is an open cover of $X$. Since $S_nX$ is normal, the cover has a refinement $\{V_1,\dots,V_k,W_0\}$ with $\overline{V}_j \subset U_j$ for each $j$ and $\overline{W}_0\subset W$ which is still a cover of $S_nX$. Since $W_0 \subset W$, it holds that $A \cap W_0=\varnothing$, such that $\{V_1,\dots,V_k\}$ is a cover of $A$. Since $\overline{V}_j\subset U_k$ for each $j$, we derive that $\{\overline{V}_1,\dots,\overline{V}_k\}$ is a cover of $A$ by closed sphere filling domains, hence $\SCcl_{n,X}(A) \leq k$, which shows the desired inequality.
\end{proof}

In the terminology of \cite[Section 1.7]{CLOT}, we have shown that $\SC_{n,X}: 2^{S_nX} \to \NN_0 \cup \{+\infty\}$ is an \emph{abstract category} on $S_nX$, also called an \emph{index function} e.g. in \cite{Zehnder}. In the following, we let $\cat_X(A)$ denote the subspace category of $A \subset X$, as defined in \cite[Definition 1.1]{CLOT}.

\begin{theorem}
If $X$ is a metrizable ANR and $A \subset S_nX$, then 
$$\cat_{S^*_nX}(A \cap S_n^*X) \leq \SC_{n,X}(A) \leq \cat_{S_nX}(A).$$ 
\end{theorem} 
\begin{proof}
It is an obvious consequence of the definition of an abstract category that the restriction of $\SC_{n,X}$ to subsets of a given subspace $A \subset S_nX$ will be an abstract category on $A$ for any such $A$. The inequality $\SC_{n,X}(A) \leq \cat_{S_nX}(A)$ is thus an immediate consequence of \cite[Theorem 1.70]{CLOT}.
Let $k := \SC_{n,X}(A)$, $\{U_1,\dots,U_k\}\subset \Dop_n(X)$ be an open cover of $A$  and put $V_j := U_j \cap S_n^*X$ for each $j \in \{1,2,\dots,k\}$. Obviously, $\{V_1,\dots,V_k\}$ is an open cover of $A \cap S_n^*X$ in $S_n^*X$ by sphere filling domains. Following the line of argument of the proof of Lemma \ref{LemmaBasedCat}, one shows that each of the inclusions $V_j  \hookrightarrow S_n^*X$ is nullhomotopic, which implies that $\cat_{S_n^*X}(A\cap 		S_n^*X) \leq k$.
\end{proof}

Since the $\SC_{n,X}$ are abstract categories, we can apply the main results of Lusternik-Schnirelmann theory to $\SC_{n,X}$ to derive upper bounds on $\SC_{n,X}(A)$ for convenient subsets $ A\subset S_nX$ in terms of the fixed-points of self-maps on $A$. For this purpose, we will use a very general topological version of the Lusternik-Schnirelmann theorem that has been established by Y. Rudyak and F. Schlenk in \cite{RudyakSchlenk}. and that may be applied to self-maps, which are not necessarily fixed-time maps of flows on the space. Given $\psi:Y \to Y$, we let $\Fix \psi$ denote its fixed-point set. 

\begin{definition}[{\cite[Definition 1.1]{RudyakSchlenk}}]
Let $X$ be a topological space and let $\varphi:X \to X$ and $f:X \to \RR$ be continuous. We say that $(f,\varphi)$ \emph{satisfies condition (D)} on $X$ if 
\begin{enumerate}[({D}1)]
	\item $f(\varphi(x))<f(x)$ for all $x \in X\setminus \Fix \varphi$; and
\item If $f$ is bounded on $A \subset X$, but $\{f(y)-f(\varphi(y)) \ | \ y \in A\}$ is not bounded away from zero, then $\bar{A} \cap \Fix \varphi \neq \varnothing$.
\end{enumerate}
\end{definition}

Before proceeding, we fix some notation that we will employ throughout the remaining manuscript. Given a topological space $Y$, a real function $f: Y\to \RR$ and $\lambda \in \RR$ we will denote the corresponding closed sublevel set and level set, by  $f^{\lambda}:= f^{-1}((-\infty,\lambda])$ and $f_\lambda := f^{-1}(\{\lambda\})$, respectively. Given an arbitrary family $(a_i)_{i  \in I}$ of elements of $\NN_0 \cup \{+\infty\}$ we further let 
$$\sum_{i \in I} a_i = \sup \Big\{ \sum_{i \in F} a_i \; \Big| \; F\subset I \text{ is finite}\Big\},$$
where we apply the obvious conventions to summations with $+\infty$.
\begin{theorem}
\label{TheoremLSRS}
Let $X$ be a metrizable ANR, $n \in \NN$ and $A \subset S_nX$. Let $F:A \to \RR$ be continuous and let $\varphi:A \to A$ be a homotopy equivalence. 
\begin{enumerate}[a)]
	\item If
\begin{itemize}
\item $F$ is bounded from below,
\item $(F,\varphi)$ satisfies condition (D) on $A$, and
\item $\SC_{n,X}(\varphi(B)) \geq \SC_{n,X}(B)$ for all closed $B \subset A$,
\end{itemize}
then for all $\lambda \in \RR$ it holds that
$$\SC_{n,X}(F^{\lambda}) \leq \sum_{\mu \in (-\infty,\lambda]} \SC_{n,X}(\Fix \varphi \cap F_\mu).$$
\item If in addition to the conditions of part a) the map $\varphi$ is homotopic to $\id_A$ and $F(\Fix \varphi)$ is finite, then 
$$\SC_{n,X}(A) \leq \sum_{\mu \in \RR}\SC_{n,X}(\Fix \varphi \cap F_\mu). $$
\end{enumerate}
\end{theorem}
\begin{proof}
Part a) is a special case of \cite[Theorem 2.3]{RudyakSchlenk}, while part b) follows by the same arguments used to prove \cite[Claim 7.2]{RudyakSchlenk}. We omit the details. 
\end{proof}

\begin{remark}
\label{RemarkLSflow}
If $X$ is a metrizable ANR and the map $\varphi$ in Theorem \ref{TheoremLSRS} is the time-1 map of a semi-flow, then the third bullet in Theorem \ref{TheoremLSRS}.a) will be satisfied by Propositions \ref{PropDeform} and \ref{PropSCclop}.
\end{remark}

The classical setting for the Lusternik-Schnirelmann theorem is the study of critical points of differentiable functions on Banach manifolds, which was established by R. Palais in \cite{PalaisLS} following the original construction by Lusternik and Schnirelmann. It follows from folklore results from differential topology, see \cite{Meier} for an explicit treatment, that if $M$ is a smooth manifold and $k \in \NN$ with $k > \frac{n}{2}$, then $H^k(S^n,M)$, the space of maps $S^n \to M$ of Sobolev class $W^{k,2}$ will be a Hilbert manifold for which the inclusion $H^k(S^n,M) \hookrightarrow \C^0(S^n,M)$ is a homotopy equivalence. Thus, if $\MM_k$ denotes the connected component of $H^k(S^n,M)$ that contains the constant maps, $\MM_k$ will be a paracompact Banach manifold for which the inclusion $\MM_k \hookrightarrow S_nM$ is a homotopy equivalence.  \\

We want to explicitly state the direct consequence of Theorem \ref{TheoremLSRS} for Banach manifolds.  Whenever $Y$ is a Banach manifold and $f$ is differentiable we let $\Crit f = \{y \in Y \ | \ df(y)=0\}$ be its set of critical points. We further let $\C^{1,1}(Y)$ be the set of those functions in $\C^1(Y)$ whose differential is locally Lipschitz-continuous.

\begin{cor}
\label{CorPalais}
Let $n \in \NN$, let $\MM\subset S_nX$ have the structure of a Hilbert manifold of class $C^2$ and let $F \in \C^{1,1}(\MM)$ be bounded from below. If $F$ satisfies the Palais-Smale condition with respect to a Riemannian metric on $\MM$, then 
$$\SC_{n,X}(F^{\lambda}) \leq \sum_{\mu \in (-\infty,\lambda]} \SC_{n,X}(\Crit F \cap F_\mu).$$
\end{cor}
\begin{proof}
Every paracompact Banach manifold is a metrizable ANR by \cite[Theorems 1 and 5]{PalaisInf}. Let $\phi_1:\MM \to \MM$ be the time-1 map of the semi-flow of a negative pseudo-gradient of $F$. The first two properties of $(F,\phi_1)$ that are required in Theorem \ref{TheoremLSRS}.a) are true by assumption. It is shown in \cite[Proposition 9.1]{RudyakSchlenk} that if $F$ satisfies the Palais-Smale condition and $\MM$ is a Hilbert manifold, then $(F,\phi_1)$ will satisfy condition (D). Together with Remark \ref{RemarkLSflow}, this shows that the claim is an application of part a) of Theorem \ref{TheoremLSRS}.
\end{proof}

In the following, we want to apply Theorem \ref{TheoremLSRS} and Corollary \ref{CorPalais} to study critical points of $O(n+1)$-invariant functions on $S_nX$. For each $n \in \NN$, the spaces $S_nX$ and $B_nX$ are equipped with continuous left $O(n+1)$-actions $$\rho_1: O(n+1) \times S_nX \to S_n X, \qquad \rho_2:O(n+1) \times B_{n+1}X \to B_{n+1}X,$$
 given by $(\rho_i(A,\gamma))(p)= \gamma(A^{-1}p)$ for $i \in \{1,2\}$ with respect to the standard $O(n+1)$-action on $S^n \subset \RR^{n+1}$ given by matrix multiplication. One immediately sees that $r_n:B_{n+1}X \to S_nX$ is $O(n+1)$-equivariant with respect to these actions. Moreover, the fixed point set of $\rho_1$ is precisely $c_n(X)$.  For simplicity, we will denote both actions by $A \cdot \gamma$ whenever it is obvious which action we are referring to. Given a subgroup $G \subset O(n+1)$, we call the restrictions of $\rho_1$ to $G \times S_nX \to S_nX$ and of $\rho_2$ to $G \times B_{n+1}X \to B_{n+1}X$ the \emph{restricted actions} and let $G \cdot \gamma$ denote the $G$-orbit of $\gamma \in S_nX$. 

\begin{prop}
\label{PropSCprimeorbit}
Let $G\subset O(n+1)$ be a closed subgroup. If $X$ is a metrizable ANR and $\gamma \in S_nX$ has trivial $G$-isotropy group with respect to the restricted action, then 
$$\SC_{n,X}(G\cdot \gamma)=1.$$
\end{prop}
\begin{proof}
Since $G$ is closed, $G \cdot \gamma$ is closed and by Proposition \ref{PropSCclop} it suffices to find a sphere filling over $G\cdot \gamma$.
Let $u \in B_{n+1}X$ be an arbitrary map with $r_n(u)=\gamma$ and define 
$$s: G\cdot \gamma \to B_{n+1}X, \qquad s(A\cdot \gamma)= A \cdot u. $$
Since the isotropy group of $\gamma$ is trivial and $r_n$ is equivariant, $s$ is a well-defined sphere filling. 
\end{proof}

We want to use this previous observation to focus on a particular class of functions.

\begin{definition}
\label{DefElike}
Let $G \subset O(n+1)$ be a closed subgroup and let $A \subset S_nX$ be a $G$-invariant subset with $c_n(X) \subset A$. Let $F:A \to \RR$ be continuous and let $\varphi: A \to A$ be a homotopy equivalence. We call $(G,A,F,\varphi)$ \emph{admissible} if 
\begin{enumerate}[(i)]
\item $F$ is a $G$-invariant function and $\Fix \varphi$ is a $G$-invariant set,
\item $F$ is constant on $c_n(X)$,
\item $(F,\varphi)$ satisfies condition (D) on $A$.
\item $F(\Fix \varphi)$ is isolated in $\RR$.
\end{enumerate}
\end{definition}

If $(G,A,F,\varphi)$ is admissible, we let $\nu(\varphi)\in \NN_0\cup \{+\infty\}$ denote the number of non-constant $G$-orbits in $\Fix \varphi$ and let $\nu(\varphi,F,\lambda)$ denote the number of non-constant $G$-orbits in $\Fix \varphi\cap F^\lambda$ for each $\lambda \in \RR$.

\begin{theorem}
\label{TheoremnuF}
Let $X$ be a metrizable ANR, $n \in \NN$ and $G \subset O(n+1)$ be a closed subgroup. Let $A\subset S_nX$ be $G$-invariant with $c_n(X) \subset A$, $F: A \to \RR$ be continuous, $\varphi: A \to A$ be a homotopy equivalence and $\lambda \in \RR$. If $(G,A,F,\varphi)$ is admissible and if every non-constant element of $\Fix \varphi \cap F^\lambda$ has trivial $G$-isotropy group, then
$$\nu(\varphi,F,\lambda) \geq \SC_{n,X}(F^{\lambda})-1 .$$
If in addition $\varphi$ is homotopic to $\id_A$ and $F(\Fix \varphi)$ is finite, then	$\nu(\varphi) \geq \SC_{n,X}(A)-1.$
\end{theorem}
\begin{proof}
If $\nu(\varphi,F,\lambda)=+\infty$, then there is nothing to show, so we will assume in the following that $\nu(\varphi,F,\lambda)<+\infty$. Since $F$ is $G$-invariant, this implies that $F(\Fix \varphi \cap F^\lambda)$ is finite. By Theorem \ref{TheoremLSRS}.a), it holds that 
$$\SC_{n,X}(F^\lambda) \leq \sum_{\mu \in (-\infty,\lambda]} \SC_{n,X}(\Fix \varphi \cap F_\mu).$$ 
$F$ is by assumption constant on $A \cap c_n(X)$ and we put $\mu_0 := F(A\cap c_n(X)) \in \RR$. If $\mu \in F(\Fix \varphi)$ with $\mu \neq \mu_0$, then $\Fix \varphi \cap F_\mu$ will be a disjoint union of $G$-orbits which are closed subsets of $A$. In this case it follows from Propositions \ref{PropSCprops}(iii) and \ref{PropSCprimeorbit} that $\SC_{n,X}(\Fix \varphi \cap F_\mu)=1$.  If $\mu=\mu_0$, then $\Fix \varphi \cap F_\mu = (\Fix \varphi\cap c_n(X)) \cup B$, where $B$ is either empty or the union of finitely many non-constant $G$-orbits. Since $\Fix \varphi \cap c_n(X)$ is closed, Proposition \ref{PropSCprops}(iii) yields $$\SC_{1,X}(\Fix \varphi \cap F_\mu) = \max\{ \SC_{n,X}(\Fix \varphi\cap c_n(X)),\SC_{1,X}(B)\}.$$ As in the other case, it follows that $\SC_{n,X}(B)=1$. Moreover, we obtain from part Proposition \ref{PropSCprops}(i) and Lemma \ref{LemmaSCconst} that $\SC_{n,X}(\Fix \varphi\cap c_n(X)) =1$.
Consequently, the above inequality implies 
$$\SC_{n,X}(F^\lambda) \leq 1+ \left|F(\Fix \varphi) \setminus \{\mu_0\} \right|.$$
Since for every $\mu \in F(\Fix \varphi)\setminus \{\mu_0\}$ the set $\Fix \varphi \cap F_\mu$, contains a non-constant $G$-orbit, the first part of the claim immediately follows. The second part follows from Theorem \ref{TheoremLSRS}.b).
\end{proof}

\begin{cor}
\label{CornuF}
Under the assumptions of Theorem \ref{TheoremnuF}, it holds that
$$\nu(\varphi,F, \lambda) \geq \cl_R\left(\im \left[\iota_\lambda^*:H^*(S_nX,c_n(X);R)\to H^*(F^{\lambda};R)\right]\right)$$
for any commutative ring $R$, where $\iota_\lambda: (F^{\lambda},\varnothing) \hookrightarrow (S_nX,c_n(X))$ denotes the inclusion of pairs. If in addition $\varphi$ is homotopic to $\id_A$ and $F(\Fix \varphi)$ is finite, then 
$$\nu(\varphi) \geq \cl_R\left(\im \left[\iota^*:H^*(S_nX,c_n(X);R)\to H^*(A;R)\right] \right),$$
where $\iota: A \hookrightarrow S_nX$ is the inclusion.
\end{cor}
\begin{proof}
This is an immediate consequence of Proposition \ref{PropSCrelCup} and Theorem \ref{TheoremnuF}.
\end{proof}

\section{$\SC_1$ and gradient flows}
\label{SectionSC1}

In this section, we want to focus on the case $n=1$ and study functionals on free loop spaces of smooth manifolds. Before doing so, we will discuss a few useful observations that can be made in this special case.

\subsection{General observations}

Let $X$ be a path-connected topological space. For $\gamma:S^1 \to X$ and $m \in \NN$ we let $\gamma^m:S^1 \to X$ denote the $m$-th iterate of $\gamma$, i.e. $\gamma^m(t) = \gamma(mt)$ for all $t \in S^1$. We further let $q_m:B^2 \to B^2$, $q_m(z)=z^m$, where we identify $B^2$ with the unit disc in $\CC$. Then $\gamma^m = \gamma \circ q_m|_{S^1}$ for each $m \in \NN$. 

\begin{prop}
If $X$ is a metrizable ANR for which $\pi_1(X)$ is torsion-free and $G \subset O(2)$ is a closed subgroup, then, with respect to the restricted $G$-action on $S_1X$,
	$$\SC_{1,X}(G\cdot \gamma) = 1\qquad \forall \gamma \in S_1X.$$
\end{prop}
\begin{proof}
For $\gamma \in c_1(X)$ this is an obvious consequence of Lemma \ref{LemmaSCconst} and Proposition \ref{PropSCprops}.(i). 
	
Let $\gamma \in S_1X \setminus c_1(X)$. Then the $G$-isotropy group of $\gamma$ is a proper closed subgroup of $O(2)$ other than $SO(2)$, i.e. either a finite cyclic group or a dihedral group.  

Let $m \in \NN$ be the biggest number for which $\gamma$ is invariant under the $\ZZ_m$-action by rotations. Since $\pi_1(X)$ is torsion-free, there exists $\beta \in S_1X$, such that $\gamma =\beta^m$. Assume first that the isotropy group of $\gamma$ is finite cyclic, hence isomorphic to $\ZZ_m$. Choose and fix a continuous map $u:B^2 \to X$ with $u|_{S^1}=\beta$ and consider $u \circ q_m:B^2 \to X$. Then $u\circ q_m|_{S^1} = \beta^m = \gamma$. 	
Using this construction, we define a map $s:G\cdot \gamma \to B_2X$, $s(A \cdot \gamma)= (u \circ q_m)(A^{-1}z)$. Since one easily sees that $q_m(B^{-1}z)=q_m(z)$ if and only if $B$ lies in the isotropy group of $\gamma$, it follows from its definition that $s$ is well-defined and satisfies $s(A \cdot \gamma)|_{S^1}=A \cdot \gamma$ for all $A \in G$. Thus, $s$ is a sphere filling and since $G\cdot \gamma$ is closed, this yields $\SC_{1,X}(G \cdot \gamma)=\SCcl_{1,X}(G\cdot \gamma) =1$ in this case, 
where we have used Proposition \ref{PropSCclop}.

If the isotropy group of $\gamma$ is not finite cyclic, it will be isomorphic to the dihedral group with $2m$ elements, i.e. in addition to the cyclic symmetry by rotations, it will hold that  $\gamma(t)=\gamma(1-t)$ for all $t \in [0,1]$. In this case choose a continuous map $u_+: B^2\cap \{ z \in \CC\;|\; \mathrm{Im}\; z \geq 0\} \to X$ with $\gamma(t) = u_+(e^{2\pi it})$ for all $t\in [0,\frac12]$ and put 
$$u:B^2 \to X, \qquad u(z) = \begin{cases}
u_+(z) & \text{if } \mathrm{Im} \; z \geq 0, \\
u_+(\bar{z}) &\text{if }\mathrm{Im} \; z < 0.
\end{cases}$$
Then $u$ is continuous with $u(e^{2\pi i t}) =\gamma(t)$ for each $t \in [0,1]$ and $u(\bar{z})=u(z)$ for all $z \in B^2$. Since $q_m(\bar{z}) = \overline{q_m(z)}$, we obtain $(u \circ q_m)(\bar{z})=(u \circ q_m)(z)$ for all $z \in B^2$, which allows us to construct a well-defined sphere flling $s: G \cdot \gamma \to B_2X$ as in the previous case. 
\end{proof}

As a consequence, we may simplify remove the trivial isotopy assumption of Theorem \ref{TheoremnuF} in the case $n=1$ and replace it by an assumption on $\pi_1(X)$ to obtain the following result, which is proven along the same lines as Theorem \ref{TheoremnuF}.

\begin{theorem}
\label{TheoremMainSC1}
Let $X$ be a metrizable ANR for which $\pi_1(X)$ is torsion-free. Let $G \subset O(2)$ be a closed subgroup and $A\subset S_1X$ be $G$-invariant with $c_1(X) \subset A$. Let $F:A \to \RR$ be continuous and let $\varphi:A \to A$ be a homotopy equivalence. Assume that $(G,A,F,\varphi)$ is admissible.
\begin{enumerate}[a)]
	\item For each $\lambda >0$ it holds that $$\nu(\varphi,F,\lambda) \geq \SC_{1,X}(F^{\lambda})-1,$$ where $\nu(F,\varphi,\lambda)$ denote the number of $G$-orbits in $\Fix \varphi \cap F^\lambda$. 
	\item If $\varphi$ is homotopic to $\id_A$ and $F(\Fix \varphi)$ is finite, then $\nu(\varphi) \geq \SC_{1,X}(A)-1.$
	\end{enumerate}
\end{theorem}

\subsection{Gradient flows of Lagrangian action functionals}

In this subsection we let $M$ be a closed connected $n$-dimensional Riemannian manifold, $n \in \NN$. We let $\Lambda M :=H^1(S^1,M)$ be the complete Hilbert manifold of maps $S^1 \to M$ that is locally modelled on the Sobolev space $H^1(S^1,\RR^n)=W^{1,2}(S^1,\RR^n)$. The Riemannian metric on $M$ further induces a Riemannian metric on $\Lambda M$. See \cite[Sections 2.3 and 2.4]{KlingRiem} for a detailed construction.

Since $\Lambda M$ is a Hilbert manifold, it is a paracompact Banach manifold, thus a metrizable ANR. Furthermore, $\C^\infty(S^1,M) \subset \Lambda M\subset \C^0(S^1,M)$ and $\Lambda M$ has the homotopy type of $\C^0(S^1,M)$. This Hilbert manifold is used to study periodic orbits of Lagrangian dynamical systems which arise as critical points of action functionals of autonomous Lagrangians on $\Lambda M$.

\begin{definition}
\begin{enumerate}[a)]
\item The \emph{action functional} of $L\in \C^\infty(TM,\RR)$ is given by 
$$\Acal_L: \Lambda M \to \RR, \qquad \Acal_L(\gamma)= \int_0^1 L(\gamma(t),\dot\gamma(t))\; dt.$$
\item We call $L\in \C^\infty(TM,\RR)$ \emph{symmetric} if $L(q,v) = L(q,-v)$ for all $q \in M$, $v \in T_qM$. 
\end{enumerate}
\end{definition}

By definition of $\Acal_L$, it is clear that $\Acal_L$ is $SO(2)$-invariant and that it is $O(2)$-invariant if $L$ is symmetric. We want to focus on an important class of Lagrangians first studied in full generality by A. Abbondandolo and M. Schwarz in \cite{AbboSchwarzCPAM} and further elaborated upon in \cite{AbboSchwarzPseudo}.  In the following, we orient our treatment on \cite[Appendix A]{FrauenfMerryPaternain}. \\

For any $x\in TM$ and $w \in T_{x}TM$ we let $w^h, w^v \in T_{x}TM$ denote the horizontal and vertical component, resp., of $w$ with respect to the Levi-Civita connection of the given metric on $M$. For any $(q_0,v_0) \in TM$ let $\nabla L(q_0,v_0)\in T_{v_0}T_{q_0}M$ denote the gradient of $T_{q_0}M \to \RR$, $v \mapsto L(q_0,v)$, at $v_0$ with respect to the given metric. Define $\nabla_qL, \nabla_vL: TM \to TM$ by $\nabla_qL(q_0,v_0):= (\nabla L(q_0,v_0))^h$ and $\nabla_vL(q_0,v_0):= (\nabla L(q_0,v_0))^v$ and put
\begin{align*}
&\nabla_{qq}L(q_0,v_0)(w) :=(D(\nabla_q L)_{(q_0,v_0)}\left[\tilde{w}\right])^h, \quad \nabla_{vq}L(q_0,v_0)(w) := (D(\nabla_q L)_{(q_0,v_0)}\left[\tilde{w}\right] )^v \ , \\
 &\nabla_{vv}L(q_0,v_0)(w) := (D(\nabla_v L)_{(q_0,v_0)}\left[\tilde{w}\right] )^v \ .
\end{align*}

\begin{definition}
$L: TM \to \RR$ is a \emph{convex quadratic-growth Lagrangian} if it satisfies the following two conditions:
\begin{enumerate}[(L1)]
 \item There exists a continuous map $l_1:M \to \RR$ such that for every $(q,v) \in TM$ it holds that
 \begin{align*}
  &\left\|\nabla_{vv}L(q,v) \right\|_{\op} \leq l_1(q),  \ \left\|\nabla_{vq}L(q,v) \right\|_{\op} \leq l_1(q)\left(1+\|v\|_q\right) , \  \left\|\nabla_{qq}L(q,v) \right\|_{\op} \leq l_1(q)\left(1+\|v\|_q^2\right),
 \end{align*}
where $\|v\|_q$ denotes the norm on $T_qM$ induced by the given Riemannian metric.
 \item There exists a continuous map $l_2:M \to (0,+\infty)$ , such that 
 \begin{equation*}
  \left<v,\nabla_{vv} L(q,v)\right>_q \geq l_2(q) \cdot \|v\|_q \quad \forall (q,v) \in TM,
 \end{equation*}
 where $\left<\cdot,\cdot \right>_q$ is the restriction of the Riemannian metric to $T_qM$.
\end{enumerate}
\end{definition} 

It is shown in \cite{AbboSchwarzPseudo} that the action functional of a convex quadratic-growth Lagrangian is of class $C^{1,1}$ and that its critical points are precisely the periodic orbits of the Euler-Lagrange vector field of $L$. In the following, we thus call critical points of $\Acal_L$ \emph{periodic orbits of $L$}. Since $\Acal_L$ is $S^1$-invariant, its critical points occur in $S^1$-orbits and in $O(2)$-orbits if $L$ is symmetric.


\begin{definition}
We call a Lagrangian $L:TM \to \RR$ \emph{of quadratic type} if $L$ is a convex quadratic-growth Lagrangian and if $L(q,0) =0$ for all $q \in M$. 
\end{definition} 

\begin{remark}
\label{RemarkActionEnergy}
The class of Lagrangians of quadratic type in particular includes a certain class of \emph{electro-magnetic Lagrangians}, namely those of the form  
$L(q,v) = \frac12\|v\|_q^2+ \alpha_q(v), $
where $\alpha \in \Omega^1(M)$ is a smooth one-form and $\| \cdot \|_q$ denotes the norm on $T_qM$ induced by the Riemannian metric. Such Lagrangians arise in electro-magnetic systems in physics, where $\alpha$ is a magnetic potential.
\end{remark}

We put $\Lambda_1 M := \Lambda M \cap S_1M$, i.e., $\Lambda_1M$ is the connected component of $\Lambda M$ that contains the constant loops. Given a convex quadratic-growth Lagrangian $L$ we let $$\Acal_{L,1}:=\Acal_L|_{\Lambda_1 M}:\Lambda_1M \to \RR$$ denote the restriction of its action functional. The critical points of $\Acal_{L,1}$ are precisely the \emph{contractible} periodic orbits of $L$. 

\begin{lemma}
\label{LemmaActionAdmiss}
Let $L:TM \to \RR$ be a Lagrangian of quadratic type. Let $$\phi: \Lambda_1M \times[0,+\infty) \to \Lambda_1M$$ be the negative gradient flow of $\Acal_{L,1}$ with respect to the metric on $\Lambda_1M$ and denote its time-1 map by $\phi_1(\gamma):=\phi(\gamma,1)$. Let $\lambda \in \RR$ and assume that the critical values of $\Acal_{L,1}$ in $(-\infty,\lambda]$ are isolated. Then:
\begin{enumerate}[a)]
\item $(SO(2),\Acal_{L,1}^\lambda,\Acal_{L,1},\phi_1)$ is admissible.
\item If $L$ is symmetric, then $(O(2),\Acal_{L,1}^\lambda,\Acal_{L,1},\phi_1)$ will be admissible. 
\end{enumerate}
\end{lemma}
\begin{proof}
First of all, we observe that the flow $\phi$ is locally well-defined since $\Acal_{L,1}$ is of class $C^{1,1}$.Since $L$ satisfies the convex quadratic-growth conditions and is bounded from below, it follows from \cite[Section 4]{Benci}, see also \cite[Proposition 3.3]{AbboSchwarzPseudo}, that $\Acal_{L,1}$ satisfies the Palais-Smale condition with respect to the given metric, which further implies that $\phi$ is defined on all of $\Lambda_1M  \times [0,+\infty)$. Moreover, it follows by definition that $\phi_1$ preserves sublevel sets of $\Acal_{L,1}$. 

As a gradient flow map of an $SO(2)$-invariant function, $\phi_1$ is $SO(2)$-equivariant, so that condition (i) of an admissible quadruple is satisfied for part a) of the lemma. Condition (i) in part b) follows in the same way, since $\Acal_{L,1}$ is $O(2)$-invariant in that case. 
Condition (ii) obviously follows from the definition of Lagrangians of quadratic type. It follows from \cite[Proposition 9.1]{RudyakSchlenk} that since $\Acal_{L,1}$ satisfies the Palais-Smale condition, $(\Acal_{L,1},\phi_1)$ satisfies condition (D) in this case, i.e. condition (iii) is satisfied. 

Since $\phi_1$ is a negative gradient flow map for $\Acal_{L,1}$, it holds that $\Fix \phi_1 = \Crit \Acal_{L,1}$. Hence, $\Acal_{L,1}(\Fix \phi_1)$ coincides with the set of critical values of $\Acal_{L,1}$, so condition (iv) is satisfied by assumption on $\lambda$.
\end{proof}

\begin{theorem}
\label{TheoremLSaction}
Assume that $\pi_1(M)$ is torsion-free and let $R$ be a commutative ring, $L:TM \to \RR$ be a Lagrangian of quadratic type, $\lambda \in (0,+\infty)$ and $\iota_\lambda: (\Acal_{L,1}^\lambda,\varnothing) \to (\Lambda_1M,c_1(M))$ be the inclusion of pairs.
\begin{enumerate}[a)]
\item Let $N_S(\lambda)$ denote the number of $SO(2)$-orbits of contractible periodic orbits of $L$ with action at most $\lambda$. Then 
$$N_S(\lambda) \geq \SC_{1,M}(\Acal_{L,1}^\lambda)-1\geq \cl_R\left(\im \left[\iota^*_\lambda:H^*(\Lambda_1M,c_1(M);R) \to H^*(\Acal_{L,1}^\lambda;R)\right] \right). $$
\item Assume that $L$ is symmetric and let $N(\lambda)$ denote the number of $O(2)$-orbits of contractible periodic orbits of $L$ with action at most $\lambda$. Then 
$$N(\lambda) \geq \SC_{1,M}(\Acal_{L,1}^\lambda)-1\geq \cl_R\left(\im \left[\iota^*_\lambda:H^*(\Lambda_1M,c_1(M);R) \to H^*(\Acal_{L,1}^\lambda;R)\right] \right). $$
\end{enumerate}
\end{theorem}
\begin{proof}
If $N(\lambda) = +\infty$, then the statement is trivial. If $N(\lambda) <+\infty$, then the set of critical values of $\Acal_{L,1}$ in $[0,\lambda]$ is finite, hence isolated. Thus, it follows immediately from combining Lemma \ref{LemmaActionAdmiss} with Theorem \ref{TheoremMainSC1}.
\end{proof}

\subsection{Gradient flows of Finsler energy functionals}
\label{SubFinsler}

As follows from Remark \ref{RemarkActionEnergy}, the study of action functionals of Lagrangians of quadratic type includes the study of energy functionals of Riemannian metrics on $M$ whose non-constant critical points are precisely the closed geodesics of the metric. In this subsection, we want to elaborate upon this observation and to generalize the result to closed geodesics of Finsler metrics on $M$. We provide their definition before continuing and refer the interested reader to \cite{BCS} or \cite{RadeNonrev} for details. 

Throughout this subsection, we let again $M$ be a closed connected Riemannian manifold. The following definition is taken from \cite{RadeHabil}.

\begin{definition}
\label{DefFinsler}
A \emph{Finsler metric on $M$} is a continuous map $F: TM \to [0,+\infty)$, such that
\begin{itemize}
\item the restriction of $F$ to $TM \setminus \zero_M$ is smooth, where $\zero_M$ denotes the image of the zero-section,
\item $F(x,\lambda v)=\lambda F(x,v)$ for all $\lambda \geq 0$, $(x,v) \in TM$,
\item $F(x,v)=0$ if and only if $(x,v)\in \zero_M$,
\item the second derivative of $T_xM \to \RR$, $v \mapsto F(x,v)^2$, is positive definite for all $x \in M$.
\end{itemize}
$F$ is called \emph{reversible} if $F(x,-v)=F(x,v)$ for every $(x,v) \in TM$. The \emph{energy functional of a Finsler metric $F$} on $M$ is given by
$$E_F: \Lambda M \to \RR,   \qquad E_F(\gamma)=\int_0^1 F(\gamma(t),\dot\gamma(t))^2 \ dt,$$ 
\end{definition}

Clearly, if $g$ is a Riemannian metric on $M$, then $F(x,v)=\sqrt{g_x(v,v)}$ will be a reversible Finsler metric on $M$. Since Finsler metrics are not required to be differentiable on the zero-section of $TM$, they do not fit into the Lagrangian framework of the previous subsection. 

Let $F$ be a Finsler metric on $M$. We observe that $E_F$ is $SO(2)$-invariant and that it will be $O(2)$-invariant if $F$ is reversible. It is well-known that $E_F$ is of class $\C^{1,1}$, see \cite[Theorem 4.1]{Mercuri}. The critical points of $E_F$ are called \emph{the closed geodesics of $F$}. By definition, $E_F$ vanishes on $c_1(M)$.

It has further been shown by F. Mercuri in \cite[Theorem 4.6]{Mercuri} that $E_F$ satisfies the Palais-Smale condition with respect to the Riemannian metric on $\Lambda M$ induced by the one on $M$. 

We consider the restriction $E_{F,1}:= E_F|_{\Lambda_1 M}: \Lambda_1 M \to \RR$. One shows, along the  lines of the proof of Lemma \ref{LemmaActionAdmiss}, that it follows from the above properties that $(SO(2),E_{F,1}^\lambda,E_{F,q},\phi_1)$ is admissible and that $(O(2),E_{F,1}^\lambda,E_{F,q},\phi_1)$ is admissible if $F$ is reversible, where $\lambda \in \RR$ and $\phi_1$ is the time-1 map of the negative gradient flow of $E_F$ with respect to the chosen Riemannian metric. The following statement is derived in complete analogy with Theorem \ref{TheoremLSaction}.

\begin{theorem}
\label{TheoremLSFinsler}
Assume that $\pi_1(M)$ is torsion-free and let $R$ be a commutative ring. Let $F: TM \to \RR$ be a Finsler metric on $M$. Let $\lambda \in (0,+\infty)$ and let $\iota_\lambda: (E_{F,1}^\lambda,\varnothing) \to (\Lambda_1M,c_1(M))$ be the inclusion of pairs.
\begin{enumerate}[a)]
\item Let $N_S(\lambda)$ denote the number of $SO(2)$-orbits of contractible closed geodesics of $F$ with energy at most $\lambda$. Then 
$$N_S(\lambda) \geq \SC_{1,M}(E_{F,1}^\lambda)-1\geq \cl_R\left(\im \left[\iota^*_\lambda:H^*(\Lambda_1M,c_1(M);R) \to H^*(E_{F,1}^\lambda;R)\right] \right). $$
\item Assume that $F$ is reversible and let $N(\lambda)$ denote the number of $O(2)$-orbits of contractible closed geodesics of $F$ with energy at most $\lambda$. Then 
$$N(\lambda) \geq \SC_{1,M}(E_{F,1}^\lambda)-1\geq \cl_R\left(\im \left[\iota^*_\lambda:H^*(\Lambda_1M,c_1(M);R) \to H^*(E_{F,1}^\lambda;R)\right] \right), $$
\end{enumerate}
\end{theorem}

Whenever $\gamma$ is a closed geodesic, all of its iterates $\gamma^m$ for $m \in \NN$ will be closed geodesics as well. In Finsler geometry one is interested in the number of closed geodesics that do not arise as iterates of one another, which we make precise in the following definition.

\begin{definition}
We call $\gamma_1,\gamma_2 \in \Lambda M$ \emph{geometrically distinct} if $\gamma_1(S^1) \neq \gamma_2(S^1)\subset M$ and \emph{positively distinct} if they are either geometrically distinct or they lie in the same $O(2)$-orbit of $\Lambda M$, but not in the same $SO(2)$-orbit. 
We call a subset of $\Lambda M$ \emph{geometrically (positively) distinct} if its elements are pairwise geometrically (positively) distinct. 
We further call $\gamma \in \Lambda M$ \emph{prime} if there are no $\beta \in \Lambda M$ and $m \geq 2$, such that $\gamma= \beta^m$. 
\end{definition}

Making use of the behavior of energy functionals under iteration, we will next derive lower bounds on the numbers of distinct contractible closed geodesics of Finsler metrics.

Given a map $f:X \to \RR$ and $\lambda \in \RR$ we let $f^{<\lambda}:=f^{-1}((-\infty,\lambda))$ denote the associated open sublevel set. 

\begin{theorem}
\label{TheoremFinslerIter}
Assume that $\pi_1(M)$ is torsion-free and let $F: TM \to [0,+\infty)$ be a Finsler metric that admits at least one non-constant contractible closed geodesic and let $\ell>0$ be the length of the shortest non-constant contractible closed geodesic of $F$ in $M$. Let $E_1:\Lambda_1M \to \RR$ be the restricted energy functional of $F$. Then:
\begin{enumerate}[a)]
\item There are at least 
$$\sup_{n \in \NN} \; \frac{1}{n}\Big(\SC_{1,M}(E_1^{<(n+1)^2\ell^2})-1\Big)$$ 
non-constant positively distinct contractible closed geodesics of $F$.
\item If $F$ is reversible, then there will be at least $\sup_{n \in \NN} \; \frac{1}{n}(\SC_{1,M}(E_1^{<(n+1)^2\ell^2})-1)$  non-constant geometrically distinct contractible closed geodesics of $F$.
\end{enumerate}
\end{theorem}
\begin{proof}
We assume that every sublevel set $E_1^\lambda$ for $\lambda \in \RR$ contains only finitely many $SO(2)$-orbits of prime closed geodesics, since otherwise the statement is obvious.

Let $\gamma \in \Lambda_1M$ be a closed geodesic of $F$ of length $\ell$.  Then $E_1(\gamma^k) = k^2\ell^2$ for each $k \in \NN$. Thus, by construction, $(n+1)^2\ell^2$ is the lowest value of $E_1$ whose level set contains an $(n+1)$-times iterated closed geodesic of $F$. Thus, if $\lambda_n>0$ is chosen such that $[\lambda_n,(n+1)^2\ell^2)$ consists of regular values only, then every closed geodesic in $E_1^{\lambda_n}$ is either prime or the $k$-fold iterate of a prime closed geodesic for some $k \leq n$. Thus, if $p_{n,S} \in \NN$ is the number of $SO(2)$-orbits prime closed geodesics in $E^{\lambda_n}$ and $p_n \in \NN$ is the number of $O(2)$-orbits, then
$$N(\lambda_n) \leq p_n \cdot n.$$
Since $[\lambda_n,(n+1)^2\ell^2)$ does not contain any critical value of $E_1$, it follows that $E_1^{<(n+1)^2\ell^2}$ deformation retracts onto $E_1^{\lambda_n}$.
It thus follows from Theorem \ref{TheoremMainSC1} and Proposition \ref{PropDeform} that 
$$p_{n,S} \geq \frac1n\Big(\SC_{1,M}(E_1^{\lambda_n})-1\Big)= \frac1n\Big(\SC_{1,M}(E_1^{<(n+1)^2\ell^2})-1\Big)$$
and, if $F$ is reversible, $p_n \geq \frac1n(\SC_{1,M}(E_1^{<(n+1)^2\ell^2})-1).$
Since two prime geodesics are geometrically distinct if and only if they do not lie in the same $O(2)$-orbit and positively distinct if and only if they do not lie in the same $SO(2)$-orbit, this shows the claim if one considers all $n \in \NN$ at once. 
\end{proof}

As in the previous subsection, we obtain a lower bound on the numbers of distinct closed geodesics in terms of cup length as well.

\begin{cor}
Assume 	$\pi_1(M)$ is torsion-free. Let $F: TM \to [0,+\infty)$ be a Finsler metric that admits a non-constant contractible closed geodesic and let $E_1:\Lambda_1 M \to \RR$ be the restricted energy functional of $F$. Let $\ell>0$ be the length of the shortest non-constant contractible closed geodesic of $F$ and let $R$ be a commutative ring. Then the number of non-constant geometrically distinct contractible closed geodesics of $F$ is at least 
$$\sup_{n \in \NN} \frac{1}{n}\cl_R\Big(\im \Big[\iota_{(n+1)^2\ell^2}^*:H^*(\Lambda_1M,c_1(M);R) \to H^*(E_1^{<(n+1)^2\ell^2};R)\Big] \Big). $$
\end{cor}

\section{Weights of cohomology classes for spherical complexities}
\label{SectionWeights}

This section's aim is to derive lower bounds for spherical complexities that extend the ones from Theorem \ref{TheoremSCcup} and Proposition \ref{PropSCrelCup}. For this purpose, we will generalize the topological underpinnings of results of M. Grant and the author from \cite{TCsymp} to general fibrations. While Grant and the author have used infinite-dimensional versions of de Rham cohomology as presented in \cite{KrieglMichor}, we will transfer the approach to singular cohomology.

We begin with a subsection recalling the relevant notions of sectional category weights. In the second subsetion, we present the general construction involving the Mayer-Vietoris sequence of fiberwise joins. Before studying cohomology weights for spherical complexities in the same manner, we will separately discuss the construction for topological complexity.

\subsection{Sectional category weight}

The notion of sectional category weight has been introduced by M. Farber and M. Grant in \cite{FarberGrantSymm} and is a straightforward generalization of the concept of category weight. The latter has been defined by E. Fadell and S. Husseini in \cite{FadellHusseini} and extended and varied by Y. Rudyak in \cite{RudyakWeight}.

In the following let $p: E \to B$ be a (Hurewicz) fibration. Given a continuous map $f: X \to B$ we let $f^*p: f^*E \to X$ denote the pullback fibration. 

\begin{definition}
Let $A$ be an abelian group and let $u\in H^*(B;A)$ with $u \neq 0$. The \emph{sectional category weight of $u$}, denoted by $\wgt_p(u)$, is the largest $k \in \NN_0$, such that $f^*u=0$ for all continuous maps $f:X \to B$ with $\secat(f^*p)\leq k$, where $X$ is any topological space.
\end{definition}

The following statement bundles various results that are either shown in \cite[Section 6]{FarberGrantSymm} or can be derived along the same lines as the corresponding results for category weight, see \cite{RudyakWeight}.

\begin{theorem}
\label{Thmwgtprops}
Let $A$ be an abelian group and let $u\in H^*(B;A)$ with $u \neq 0$.
\begin{enumerate}[(1)]
\item If $\wgt_p(u)=k$, then $\secat(p) \geq k+1$.
\item $\wgt_p(u)\leq \deg u$, where $\deg$ denotes the cohomological degree of $u$.
\item If $f:X \to B$ is continuous and $f^*u \neq 0 \in H^*(X;A)$, then $\wgt_{f^*p}(f^*u) \geq \wgt_p(u)$.
\item For each $n \geq 2$ let $p_n: E^{*n}= E * \dots *E\to B$ denote the $n$-fold iterated fiberwise join of $p$ with itself and put $p_1:=p$. If $p_n^*u=0\in H^*(E^{*n};A)$, then $\wgt_p(u)\geq n$ and 
$$\wgt_p(u)= \sup\{ n \in \NN \ | \ p_n^*u=0\}.$$
\end{enumerate}
\end{theorem}

Another important property of sectional category weight is its superadditivity with respect to products. 

\begin{theorem}[{\cite[Proposition 32]{FarberGrantSymm}}]
\label{Thmwgtcup}
Let $R$ be a commutative ring and let $u_1,\dots,u_n \in H^*(B;R)$ with $u_1 \cup u_2 \cup \dots \cup u_n \neq 0$. Then
$$\wgt_p(u_1 \cup u_2 \cup \dots \cup u_n) \geq \sum_{i=1}^n \wgt_p(u_i).$$
\end{theorem}

\begin{remark}
Both Theorem \ref{Thmwgtprops} and Theorem \ref{Thmwgtcup} can be generalized to cohomology with local coefficient systems. In the light of the applications that are studied in this article, we refrain from presenting these generalizations.
\end{remark}

\subsection{The Mayer-Vietoris sequence of fiberwise joins}
\label{SubGen}

Let $p: E \to B$ be a (Hurewicz) fibration with fiber $F$. The fiberwise join $p_2:E_2 := E * E \to B$ of $p$ with itself can be described in the following way: Let $Q$ denote the pullback of $E \stackrel{p}{\rightarrow} B \stackrel{p}{\leftarrow} E $ and let
\begin{equation}
\label{pullback}
\begin{CD}
 Q @>{r_1}>> E \\
 @V{r_2}VV @V{p}VV \\
 E @>{p}>> B
 \end{CD}
\end{equation}
be a pullback diagram. $E_2$ is given as a homotopy pushout of $E \stackrel{r_1}{\leftarrow} Q \stackrel{r_2}{\rightarrow} E$ and we let $j_1,j_2: E \to E_2$ and $p_2: E_2 \to B$ be maps for which the following diagram commutes up to homotopy:
\begin{equation*}
 \xymatrix{
Q \ar[r]^{r_1} \ar[d]_{r_2}& E \ar[d]_{j_1} \ar@/^/[ddr]^{p} & \\
E \ar[r]^{j_2} \ar@/_/[rrd]^{p} & E_2 \ar[dr]^{p_2} & \\
& & B.
}
\end{equation*}
Here, $p_2:E_2 \to B$ can be chosen so that $p_2$ is a fibration with fiber $F *F$, the join of $F$ with itself.

See e.g. \cite[Section 2]{FGKV} for a detailed construction of these fiberwise joins. As a homotopy pushout, $E_2$ fits into a long exact Mayer-Vietoris sequence, see \cite[Chapter 11]{StromClassical}, which is the sequence in the top row of the diagram
$$\xymatrix{
 \dots \ar[r] & H^{k-1}(Q;A) \ar[r]^{\delta} & H^k(E_2;A) \ar[r]^{j_1^* \oplus j_2^*\qquad\ } & H^k(E;A)\oplus H^k(E;A) \ar[r]^{\qquad\ r_1^*-r_2^*} & H^k(Q;A) \ar[r] & \dots \\
  & & H^k(B;A) \ar[u]^{p_2^*} \ar[ur]_{p^* \oplus p^*} & & & 
}$$
for any abelian group $A$.

Let $k \geq 2$ and let $\sigma \in \ker [p^*: H^k(B;A)\to H^k(E;A)]$. The commutativity of the lower triangle and the exactness of the Mayer-Vietoris sequence in the above diagram yield
$$p_2^*\sigma \in \im \left[\delta: H^{k-1}(Q;A) \to H^k(E_2;A) \right].$$
To find conditions yielding $p_2^*\sigma=0$ and thus $\wgt_p(\sigma) \geq 2$, it thus suffices to find a class $\alpha_\sigma \in H^{k-1}(Q;A)$ with $\delta(\alpha_\sigma) = p_2^*\sigma$ and to find conditions ensuring that $\alpha_\sigma=0$. In the following proposition, we do the first step by constructing a representative of $\alpha_\sigma$.  

\begin{prop}
\label{PropMV}
Let $k \geq 2$ and let $\sigma \in \ker \left[p^*: H^k(B;A)\to H^k(E;A)\right]$. Let $c \in C^k(B;A)$ be a cocycle representing $\sigma$. If $\psi_1,\psi_2 \in C^{k-1}(E;A)$ satisfy $d\psi_1 =d\psi_2= p^*c$, then $a_\sigma \in C^{k-1}(Q;A)$, 
$$a_\sigma = r_1^*\psi_1 - r_2^*\psi_2, $$
is a cocycle with $\delta([a_\sigma])= p_2^*\sigma$.
\end{prop}
\begin{proof}
This follows from the definition of the Mayer-Vietoris boundary map. One computes that if $\psi_1$ and $\psi_2$ have the properties from the statement then 
$$da_\sigma = d(r_1^*\psi_1) - d(r_2^*\psi_2) = r_1^*(d\psi_1)-r_2^*(d \psi_2) = r_1^*p^*c -r_2^*p^*c = 0, $$
using the commutativity of \eqref{pullback}.
\end{proof}
The following statement summarizes the discussion of this subsection.
\begin{cor}
\label{CorGenwgt2}
Let $k \geq 2$, let $\sigma \in \ker \left[p^*: H^k(B;A)\to H^k(E;A)\right]$ and let $a_\sigma \in C^{k-1}(Q;A)$ be a cocycle constructed as in Proposition \ref{PropMV}. If $[a_\sigma]=0$, then $\wgt_p(\sigma) \geq 2$. 
\end{cor}


\subsection{TC-weight}
\label{SubTCwgt}
Throughout this subsection let $X$ be a path-connected Hausdorff space. As mentioned above, the topological complexity of $X$ is given by
$$\TC(X) = \secat(\pi: PX \to X \times X, \; \gamma \mapsto (\gamma(0),\gamma(1))),$$
where $PX =\C^0([0,1],X)$. As discussed in \cite{TCsymp}, the free loop space $LX = \C^0(S^1,X)$ is a pullback of $PX\stackrel{\pi}{\rightarrow}X \times X \stackrel{\pi}{\leftarrow} PX$. More precisely, with  $r_1,r_2:LX \to PX$ given by $(r_1(\alpha))(t)=\alpha(\frac{t}{2})$, $(r_2(\alpha))(t)=\alpha(1-\frac{t}{2})$, for all $t \in [0,1]$, $\alpha \in LX$, the above pullback diagram commutes for $p=\pi: PX \to X \times X$ and $Q=LX$.

Let $A$ be an abelian group and $\sigma \in H^*(X \times X;A)$. As established by M. Farber,  $\sigma \in \ker \pi^*$ if and only if $\sigma$ is a zero-divisor, i.e. if $\Delta_X^*\sigma=0$, where $\Delta_X:X \to X \times X$ is the diagonal map.

Let $p_2:P_2X \to X \times X$ denote the fiberwise join of $\pi$ with itself, $k \geq 2$ and $\sigma \in H^k(M \times M;A)$ be a zero-divisor. According to the above construction, to show that $p_2^*\sigma=0$, it suffices to find $\alpha_\sigma \in H^{k-1}(LX;A)$ with $\delta(\alpha_\sigma) = p_2^*\sigma$ and to find conditions for which $\alpha_\sigma=0$. 

Consider $e_0: PX \to X$, $\gamma \mapsto \gamma(0)$. Then $\pi: PX \to X \times X$ is homotopic to $D_0: PX \to X \times X$, $D_0 := \Delta_X \circ e_0$, via $h: PX \times [0,1] \to X \times X$, $h(\gamma,t)=(\gamma(0),\gamma(t))$.

Let $K: C^*(X\times X;A) \to C^{*-1}(PX;A)$ be the chain homotopy induced by $h$, so that
$$\pi^* - D_0^* = d\circ K + K \circ d,$$
where $d$ denotes the respective codifferential. Let $\sigma \in H^k(X \times X;A)$ be a zero-divisor and $\varphi \in C^k(X \times X;A)$ be a cocycle representing $\sigma$. Then
$$\pi^*\varphi - D_0^*\varphi = d(K(\varphi)), $$
since $\varphi$ is a cocycle. Since $\sigma$ is a zero-divisor, there exists $\rho \in C^{k-1}(X;A)$ with $\Delta_X^*\varphi = d\rho$, which yields $D_0^*\varphi = e_0^*\Delta_X^*\varphi = e_0^*d\rho$ and thus 
$$\pi^*\varphi = d(K(\varphi)+ e_0^*\rho).$$
Applying Proposition \ref{PropMV}, we obtain that $\delta([a_\sigma])=p_2^*\sigma,$ where $a_\sigma \in C^{k-1}(LX;A)$ is given by
\begin{equation}
\label{Eqasigma}
a_\sigma = r_1^*(K(\varphi)+e_0^*\rho)-r_2^*(K(\varphi)+e_0^*\rho)= r_1^*K(\varphi) - r_2^*K(\varphi),
\end{equation}
since $e_0 \circ r_1 = e_0 \circ r_2$, as the reader can easily check. 

\begin{lemma}
\label{LemmaK}
Let $\iota \in C_1([0,1])$, $\iota = [\id: [0,1]\to [0,1]]$, be the fundamental chain of $[0,1]$. Then the above chain homotopy $K: C^*(X\times X;A) \to C^{*-1}(PX;A)$ is given by 
$$K(\varphi) = (h^*\varphi)/\iota,$$
where $\cdot / \cdot$ denotes the slant product $$/ : C^*(PX\times [0,1];A) \otimes C_*([0,1];\ZZ) \to C^*(PX;A).$$
\end{lemma}
\begin{proof}
Let $\varphi \in C^k(X \times X;A)$ for some $k \in \NN$. For each $c \in C_{k-1}(PX)$, it holds by definition that
$$((h^*\varphi)/\iota)(c)= \varphi(h_*(c \times \iota)).$$
But the right-hand expression already coincides with $K(\varphi)$, see e.g. the proof of \cite[Theorem IV.16.4]{Bredon}. (Explicitly, one may show the claim by writing out the two sides in terms of an Eilenberg-Zilber map and a prism operator, resp.)
\end{proof}

In the following we let $\tau \in C_1(S^1)$ denote the singular 1-cycle defined by the quotient space projection $q:[0,1] \to S^1=[0,1]/\{0,1\}$ and denote its homology class by $[S^1]\in H_1(S^1;\ZZ)$.

\begin{prop}
\label{PropMVbdy}
Let $\sigma \in H^k(X \times X;A)$ with $k \geq 2$ be a zero-divisor and let $\ev: LX \times S^1 \to X \times X$, $\ev(\alpha,t)=(\alpha(0),\alpha(t))$. The cohomology class of $a_\sigma \in C^{k-1}(LX;A)$  from equation \eqref{Eqasigma} is given by 
$$[a_\sigma] = \ev^*\sigma/[S^1]. $$
\end{prop}
\begin{proof}
Consider the singular $1$-simplices $f_1,f_2: [0,1] \to S^1$, $f_1(t)= q(\tfrac{t}{2})$, $f_2(t)=q(1-\tfrac{t}{2})$, where $q:[0,1] \to S^1$ denotes the quotient space projection. Clearly, $\tau = f_1 - f_2 + \partial g$ for some $g \in C_2(S^1)$, where $\partial$ denotes the singular differential. Moreover, it is easy to see that 
$$\ev \circ (\id_{LX}\times f_i) = h \circ (r_i \times \id_{[0,1]}) \qquad \text{for } i \in \{1,2\}.$$  
Let $\varphi \in C^k(X \times X;A)$ be a cocycle representing $\sigma$. By bilinearity of the slant product and by its behavior with respect to differentials, see \cite[p.287]{Spanier}, we derive
\begin{align*}
\ev^*\varphi/\tau 
&=\ev^*\varphi/f_1 - \ev^*\varphi/f_2 + (-1)^{k-1} d(\ev^*\varphi/g).
\end{align*}
We compute for $i \in \{1,2\}$ and $c \in C_{k-1}(LX)$ that
\begin{align*}
(\ev^*\varphi/f_i)(c) &= (\ev^*\varphi)( c \times f_i )= \varphi(\ev_*( \id_{LX}\times f_i)_*(c \times\iota)) = \varphi(h_*(r_i \times \id_{[0,1]})_*(c\times\iota))\\
&=(r_i\times \id_{[0,1]})^*h^*\varphi(c\times\iota) = (((r_i\times \id_{[0,1]})^*h^*\varphi)/\iota)(c) =(r_i^*K(\varphi))(c),
\end{align*}
using Lemma \ref{LemmaK} and property 1 from \cite[p. 287]{Spanier} about the compatibility of the slant product with pullbacks. Using the previous computations and \eqref{Eqasigma}, we derive that
$$\ev^*\sigma/[S^1]= \left[\ev^*\varphi/\tau\right] = \left[r_1^*K(\varphi) - r_2^*K(\varphi)\right]=[a_\sigma].$$
\end{proof}

\begin{theorem}
\label{TheoremTC}
Let $R$ be a commutative ring and put 
$$E(X;R) := \{ \sigma \in H^*(X\times X;R) \ | \ \ev^*\sigma/[S^1] = 0 \},$$ 
where $[S^1] \in H_1(S^1;R)$ denotes the fundamental class. Then 
$$\TC(X) \geq 2\cl_R E(X;R)+1.$$	
\end{theorem}
\begin{proof}
Combining  Proposition \ref{PropMVbdy} with Corollary \ref{CorGenwgt2} shows that $\wgt_\pi(\sigma) \geq 2$ for all $\sigma \in E(X;R)$. Thus, the inequality follows from Theorem \ref{Thmwgtcup} and Theorem \ref{Thmwgtprops}.(1). 
\end{proof}

In the upcoming sections, we will establish more tangible criteria for a class in $H^*(X \times X;R)$ to lie in $E(X;R)$. 

\begin{remark}
\label{Remarkcwgt2}
An analogue of Theorem \ref{TheoremTC} holds in the case of LS category. The LS category of a space $X$ is given by $\cat (X) = \secat(p_0:P_{x_0}X \to X)$, where $p_0(\gamma)=\gamma(1)$. It is well-known, see \cite[Example 1.61]{CLOT}, that the fiberwise join of $p_0$ with itself is homotopy-equivalent to $\Sigma\Omega X$, where $\Sigma$ denotes reduced suspension, and with  $e: \Sigma \Omega X \to X$, $e([\gamma,t])=\gamma(t)$,  one shows that $\sigma \in \widetilde{H}^*(X;R)$ satisfies $\wgt_{p_0}(\sigma) \geq 2$ if $u \in \ker [e^*:H^*(X;R)\to H^*(\Sigma \Omega X;R)]$. From this observation, one derives that in terms of the map $\ev_0: \Omega X \times S^1 \to X$, $\ev_0(\gamma,t)=\gamma(t)$, it follows with $E_0(X;R) := \{\sigma \in H^*(X;R)\; | \; \ev_0^*\sigma/[S^1]=0\}$ that
$$\cat(X) \geq 2\cl_R E_0(X;R)+1.$$
\end{remark}

\subsection{Spherical complexity weights}
\label{SubSCwgt}
Let $n \in \NN_0$ and $X$ be a path-connected Hausdorff space. We recall that $\SC_n(X) = \secat(r_n :B_{n+1}X \to S_{n}X).$ and want to carry out the Mayer-Vietoris construction from subsection \ref{SubGen} for the fibration $r_n$. The pullback $Q$ of $B_{n+1}X \stackrel{r_n}{\rightarrow}S_{n}X \stackrel{r_n}{\leftarrow} B_{n+1}X$ is explicitly given by 
$$Q = \left\{(f_1,f_2) \in B_{n+1}X \times B_{n+1}X \ | \ f_1|_{S^{n}}=f_2|_{S^{n}} \right\}.$$
Viewing $(f_1,f_2) \in Q$ as maps from the two hemispheres of $S^{n+1}$ which coincide on the equator, we obtain a homeomorphism $Q \approx \C^0(S^{n+1},X)$. The restriction maps $R_1,R_2:\C^0(S^{n+1}, X) \to B_{n+1}X$ to the two hemispheres make the following pullback diagram commutative:
$$\begin{CD}
\C^0(S^{n+1}, X) @>{R_2}>> B_{n+1} X  \\
@V{R_1}VV @V{r_n}VV \\
B_{n+1} X @>{r_n}>> S_{n} X.
\end{CD}$$
Following the construction from subsection \ref{SubGen}, the fiberwise join of $r_n:B_{n+1} X \to S_{n}X$ with itself, which we denote by
$$r_{n,2}: B_{n+1,2}X \to S_{n}X,$$
yields a long exact Mayer-Vietoris sequence of the form
$$\xymatrix{
 \dots \ar[r] & H^{k-1}(\C^0(S^{n+1},X);A) \ar[r]^{\quad \delta} & H^k(B_{n+1,2}X;A) \ar[r]^{j_1^* \oplus j_2^*\quad } & \bigoplus_{i=1}^2 H^k(B_{n+1}X;A) \ar[r]^{\qquad\ \ \  R_1^*-R_2^*} & \dots \\
  & & H^k(S_{n}X;A) \ar[u]^{r_{n,2}^*} \ar[ur]_{r_n^* \oplus r_n^*} & & & 
}$$
for any abelian group $A$. Let $h_n: B_{n+1}X \times [0,1] \to S_nX$ be defined as in the proof of Theorem \ref{TheoremSCcup}. Since $X$ is Hausdorff, $h_n$ is a homotopy from $r_n$ to $D_n := c_n \circ e_n$. 
Let $$k_n:C^*(S_{n}X;A) \to C^{*-1}(B_{n+1}X;A)$$ be the chain homotopy induced by $h_n$, so that $r_n^* - D_n^* = d \circ k_n + k_n \circ d$.

Let $u \in H^k(S_nX;A)$ with $\wgt_{r_n}(u) \geq 1$ be such that $u \in \ker D_n^*=\ker r_n^*$. Let $\varphi \in C^k(S_{n}X;A)$ be a cocycle representing $u$, such that 
$$r_n^*\varphi - D_n^*\varphi = (d\circ k_n)(\varphi).$$
To employ the method from subsection \ref{SubGen}, we need to study the cocycle
$$a_u \in C^{k-1}(\C^0(S^{n+1},X);A), \qquad a_u = R_1^*k_n(\varphi) - R_2^*k_n(\varphi),$$
introduced in Proposition \ref{PropMV}. We consider the maps
$$F_{n,1},F_{n,2}: \C^0(S^{n+1},X)\times [0,1] \to S_{n}X, \qquad F_{n,i} = h_n \circ (R_i\times \id_{[0,1]}).$$
One checks without difficulty that $F_{n,1}(\gamma,0)= F_{n,2}(\gamma,0)$ and $F_{n,1}(\gamma,1)=c_n(\gamma(1))=F_{n,2}(\gamma,1)$ for all $\gamma \in C^0(S^{n+1},X)$. Hence, $F_{n,1}$ and $F_{n,2}$ induce a map 
$$F_n: \C^0(S^{n+1},X)\times S^1 \to S_{n}X,  \qquad F_n(\gamma,t) = \begin{cases}
F_{n,1}(\gamma,1-2t) & \text{if } t \in [0,\frac12], \\
F_{n,2}(\gamma,2t-1) & \text{if } t \in (\frac12,1].
\end{cases}$$

\begin{prop}
\label{PropMVSC}
Let $u \in H^k(S_{n}X;A)$ with $k \geq 2$, with $u \in \ker c_n^*$. Then the cohomology class of $a_u \in C^{k-1}(\C^0(S^{n+1},X);A)$ is given by
$$[a_u] = -F_n^*u/[S^1].$$
\end{prop}
\begin{proof}
Consider again the singular $1$-simplices $f_1,f_2: [0,1] \to S^1$, $f_1(t) = q(\frac{t}{2})$, $f_2(t)=q(1-\frac{t}{2})$, with $q:[0,1]\to S^1$ denoting the quotient space projection. We compute that
$$F_n(\gamma,f_i(t)) = F_{n,i}(\gamma,1-t)=h_n(R_i(\gamma),1-t) \quad \forall i \in \{1,2\}, \; t \in [0,1], \; \gamma \in \C^0(S^{n+1},X).$$
 With $s:[0,1] \to [0,1]$, $s(t)=1-t$, we rewrite this as
$$F_n \circ (\id_{\C^0(S^{n+1},X)}\times f_i)= h_n \circ (R_i \times s) \qquad \forall i \in \{1,2\}.$$
As in the proof of Lemma \ref{LemmaK}, one shows that the chain homotopy $k_n$ is given by $k_n(\varphi) = h_n^*(\varphi)/\iota.$
Using this observation and letting $\left<\cdot,\cdot \right>$ denote chain-level Kronecker products, we compute for $i \in \{1,2\}$ and $c \in C_{k-1}(\C^0(S^{n+1},X);A)$ that
\begin{align*}
\left<R_i^*k_n(\varphi),c \right>&= \left< \varphi, (h_n)_*((R_i)_*(c)\times \iota\right> = \left<\varphi,(h_n \circ (R_i\times s))_*(c \times s_*\iota) \right> \\
&= \left<\varphi,(F_n \circ (\id_{\C^0(S^{n+1},X)}\times f_i))_*(c \times s_*\iota) \right> = \left<F_n^*\varphi/((f_i)_*s_*\iota),c\right>. 
\end{align*}
Consequently, 
$$a_u = R_1^*k_n(\varphi) - R_2^*k_n(\varphi) = F_n^*\varphi/((f_1)_*s_*\iota - (f_2)_*s_*\iota).$$
It is easy to see that $(f_1)_*s_*\iota - (f_2)_*s_*\iota$ is a cycle representing the class $-[S^1]$, so passing to cohomology shows that $ [a_u] =F_n^*u/(-[S^1])= - F_n^*u/[S^1]$.
\end{proof}

We denote the sectional category weight of $u \in H^*(S_nX;A)$ with respect to $r_n$ by $$\wgt_n(u):= \wgt_{r_n}(u).$$

\begin{theorem}
\label{ThmSCnEn}
Let $R$ be a commutative ring and put 
$$E_n(X;R) := \left\{ \sigma \in H^*(S_nX;R) \ \middle| \ \deg \sigma \geq 2, \ \  F_n^*\sigma/[S^1]=0\right\}.$$
Then 
$$\SC_n(X) \geq 2 \cl_R(E_n(X;R))+1.$$
\end{theorem}
\begin{proof}
Combining Proposition \ref{PropMVSC} with Corollary \ref{CorGenwgt2} applied to $p=r_n$ yields $\wgt_n(\sigma) \geq 2$ for all $\sigma \in E_n(X;R)$. The inequality thus follows from Theorem \ref{Thmwgtcup} and part 1 of Theorem \ref{Thmwgtprops}.
\end{proof}

We want to establish an extension of Theorem \ref{ThmSCnEn} that considers the subspace complexities $\SC_{n,X}(A)$ in a similar way. 

\begin{theorem}
\label{ThmScnEnrel}
Let $R$ be a commutative ring and let $E_n(X;R)$ be defined as in Theorem \ref{ThmSCnEn}. Let $A \subset S_nX$ and let $i_A:A \hookrightarrow S_nX$ denote the inclusion. Then 
$$\SC_{n,X}(A) \geq 2 \cl_R \left(i_A^*(E_n(X;R)) \right)+1.$$
\end{theorem}
\begin{proof}
We recall that $\SC_{n,X}(A) = \secat(i_A^*r_n)$. If $u \in H^*(S_nX;R)$ satisfies $i_A^*u \neq 0 \in H^*(A;R)$, then it will follow from part 3 of Theorem \ref{Thmwgtprops} that $\wgt_{i_A^*r_n}(i_A^*u) \geq \wgt_n(u)$. Hence, $\wgt_{i_A^*r_n}(i_A^*u) \geq 2$ for each $u \in E_n(X;R)$ with $i_A^*u\neq 0$ and the claim follows in analogy with Theorem \ref{ThmSCnEn}.
\end{proof}

\section{Geometric criteria for cohomology classes of weight two}
\label{SectionGeomweight}
The conditions we derived in Section \ref{SectionWeights}  for $u \in H^*(S_nX;A)$ to have spherical complexity weight two or higher are a priori not particularly easy to check. In this section, we will consider special types of such cohomology classes for which these conditions become much more explicit. 

The results of this section generalize existing results for closed symplectic manifolds $(M,\omega)$. from \cite{TCsymp}. There, Grant and the author have shown that if $(M,\omega)$ is symplectically atoroidal, i.e. if $f^*[\omega]=0$ for each continuous $f:T^2 \to M$, then $\wgtTC([\bar\omega])=2$, where $\bar\omega:= 1 \times \omega - \omega \times 1 \in \Omega^2(M)$. This directly implies that $\TC(M)=2\dim M +1$.

Our results in this section generalize these observations to $n$-spherical complexitites for arbitrary $n \in \NN$. Given a cohomology class $u \in H^k(X;R)$ with $k \geq n+2$, we will associate a class $Z_n(u) \in H^{k-n}(S_nX;R)$ with it and establish a geometric criterion on $u$ generalizing the above atoroidality condition, which will yield that $Z_n(u) \in E_n(X;R)$ and thus $\wgt_n(Z_n(u))\geq 2$.

Similar to the previous section, we state the results for $\TC$-weight in a separate subsection.

\subsection{TC-weights}
\label{SubGeomTCwgt}

We put $\wgtTC:= \wgt_\pi$ for the $\TC$-weight of a cohomology class. We want to focus on a particular type of zero-divisors in this section. Namely, for $u \in H^*(X;R)$, where $R$ is a commutative ring, we consider
$$\bar{u} \in H^*(X \times X;R), \qquad \bar{u} := 1 \times u - u \times 1 ,$$
with $1 \in H^0(X;R)$ being the unit element.
Evidently, any such $\bar{u}$ is  a zero-divisor. We consider the right $S^1$-action on $LX$ by rotations, i.e. the map $\rho: LX  \times S^1 \to LX$, $(\rho(\alpha,s))(t) = \alpha(t+s)$, and let
$$I: C^*(LX;R) \to C^{*-1}(LX;R), \quad I(c) = \rho^*c/\tau,$$
where $\tau \in C_1(S^1)$ again denotes a representative of the fundamental class of $S^1$. Since $\tau$ is a cycle, $I$ is a chain map and we let $\bar{I}:H^*(LX;R) \to H^{*-1}(LX;R)$ denote the map induced by $I$.

\begin{prop}
\label{PropRotwgt}
Let $u \in H^k(X;R)$ with $k \geq 2$ and $u \neq 0$. If 
$$u \in \ker \left[\bar{I} \circ e_0^*:H^k(X;R) \to H^{k-1}(LX;R) \right]$$
then $\wgtTC(\bar{u})\geq 2$.
\end{prop}
\begin{proof}
Let $\varphi\in C^k(X;R)$ be a representative of $u$ and put $\bar\varphi := 1 \times \varphi - \varphi \times 1 \in C^k(X\times X;R)$. We want to investigate the cocycle $a_{\bar{u}} = \ev^*\bar\varphi/\tau \in C^{k-1}(LX;R)$. 
One checks without difficulties that $\ev = (e_0 \circ \pr_1) \times (e_0 \circ \rho)$, where $\pr_1:LX \times S^1 \to LX$ is the projection onto the first factor. Consequently, $\ev^*\bar{\varphi}= \ev^*(1 \times \varphi) - \ev^*(\varphi \times 1)= \rho^*e_0^*\varphi - \pr_1^*e_0^*\varphi$, which yields
$$a_{\bar{u}} = \ev^*\bar{\varphi}/\tau = \rho^*e_0^*\varphi/\tau - \pr_1^*e_0^*\varphi/\tau = I(e_0^*\varphi)-\pr_1^*e_0^*\varphi/\tau.$$
Since $(e_0)_*(\pr_1)_*\tau$ is easily seen to be a degenerate chain, it follows that $(\bar{I} \circ e_0^*)(u) = [a_{\bar{u}}]$ and Corollary \ref{CorGenwgt2} applied to the path fibration $\pi$ shows the claim.
\end{proof}

\begin{cor}
\label{Corslant}
Let $e: LX\times S^1 \to X$, $e(\alpha,t)= \alpha(t)$, and let $u\in H^k(X;R)$ with $k \geq 2$ and $u \neq 0$. If $e^*u/[S^1]=0$, then $\wgtTC(\bar{u})\geq 2$.
\end{cor}
\begin{proof}
Since $e= e_0 \circ \rho$, it holds that $e^*u/[S^1]= \rho^*(e_0^*u)/[S^1]= (\bar{I}\circ e_0^*)(u)$, so the first part of the claim follows from Proposition \ref{PropRotwgt}. 
\end{proof}

Let $R$ be a commutative ring with unit and assume that $R$ is not a field of characteristic two. We recall that a class $u \in H_k(Y;R)$ is called \emph{realisable}, where $Y$ is a topological space, if there exist a closed oriented $k$-dimensional manifold $P$ and a continuous map $f:P \to Y$, such that $u = f_*[P]$, where $[P] \in H_k(P;R)$ denotes the $R$-fundamental class of $P$. We call $H_k(Y;R)$ \emph{realisable} if it is generated as an $R$-module by realisable classes. If $R$ is a field of characteristic two, one drops the requirement that $P$ has to be oriented and formulates the notion of realisable classes in the same way. It was shown by R. Thom in \cite{ThomQuelques}, that $H_*(Y;\QQ)$ and $H_k(Y;\ZZ_2)$ are realisable for any space $Y$. We will use this observation in several proofs in this section.

\begin{prop}
\label{Propsigmapull}
Let $k \in \NN$ with $k\geq 2$ and let $u \in H^k(X;\QQ)$ with $u \neq 0$. If $f^*u =0$ for all continuous maps $f:P\times S^1 \to X$, where $P$ is an arbitrary closed oriented $(k-1)$-dimensional manifold, then $\wgtTC(\bar{u}) \geq 2$.
\end{prop}

\begin{proof}
We want to show that under the given assumptions it holds that $e^*u/[S^1]=0$ and apply Corollary \ref{Corslant}. Let $\varphi \in C^k(X;\QQ)$ be a representative of $u$ and consider $e^*\varphi/\tau \in C^{k-1}(LX;\QQ)$. Since we consider field coefficients, to show that $[e^*\varphi/\tau]=0$ it suffices to show that its Kronecker pairing with all elements of $H_{k-1}(LX;\QQ)$ vanishes. 

Let $P$ be a closed oriented $(k-1)$-dimensional manifold and let $f:P \to LX$ be continuous. Let $\tilde{f}: P \times S^1 \to X$ be the adjoint map of $f$ and let $\left< \cdot, \cdot\right>$ denote both the chain-level and the homology-level Kronecker pairing and let $\sigma_P \in C_{k-1}(P;\QQ)$ a fundamental cycle of $P$. Using the simple observation that $e \circ (f \times \id_{S^1})=\tilde{f}$, we derive by definition of the slant product that
$$\left<e^*\varphi/\tau,f_*\sigma_P \right>=\left<e^*\varphi,f_*\sigma_P\times \tau\right>
= \left<\varphi, e_*(f\times\id_{S^1})_*(\sigma_P \times \tau) \right> = \left<\tilde{f}^*\varphi, \sigma_P \times \tau \right>.$$
Passing to (co)homology, we obtain $$\left<e^*u/[S^1], f_*[P]\right>=\left<[e^*\varphi/\tau],[f_*\tau_P] \right>=\left<\tilde{f}^*u,[ P\times S^1] \right>=0$$ by assumption on $u$. Hence, since $H_{k-1}(LX;\QQ)$ is realisable, $e^*u/[S^1]=0$ and the claim follows from Corollary \ref{Corslant}.
\end{proof}
In the case of $k=2$ in Proposition \ref{Propsigmapull},  the claim is valid for arbitrary coefficient rings, since the Kronecker pairing is always non-degenerate in degree $1$ and $H_1(LX;R)$ is realizable for any coefficient ring by the Hurewicz theorem.
\begin{cor}
\label{CorAtor}
Let $R$ be a commutative ring and $u \in H^2(X;R)$ with $u \neq 0$. If $u$ is \emph{atoroidal}, i.e. if $f^*u =0$ for every continuous map $f:T^2=S^1 \times S^1 \to X$, then $\wgtTC(\bar{u})=2$.
\end{cor}

\begin{remark}
Corollary \ref{CorAtor} generalizes the main result from \cite{TCsymp} sketched in the beginning of this subsection from real cohomology to arbitrary coefficient rings.
\end{remark}

We will next discuss a result similar to Proposition \ref{Propsigmapull} for spaces of higher connectivity.

\begin{prop}
Let $\sigma \in H^k(X;\QQ)$ with $\sigma \neq 0$, $k \geq 2$. If $X$ is $(k-1)$-connected and if $f^*\sigma =0$ for all continuous maps $f:S^1 \times S^{k-1} \to X$, then $\wgtTC(\bar\sigma)\geq 2$.
\end{prop}

\begin{proof}
The long exact homotopy sequence of the fibration $e_0:LX \to X$ shows that if $X$ is $(k-1)$-connected, then $LX$ will be $(k-2)$-connected. Thus, the Hurewicz theorem yields that $H_{k-1}(LX;\K)$ is generated by classes of the form $f_*[S^{k-1}]$, where $f:S^{k-1} \to LX$ is continuous. One continues as in the proof of Proposition \ref{Propsigmapull}, restricting to the case of $P=S^{k-1}$.
\end{proof}

Results analogous to the previous ones, but for category weights, can be derived in a more straightforward way from the observations outlined in Remark \ref{Remarkcwgt2}. Thus, we state the results, but omit their proofs.  Let $\cwgt:= \wgt_{p_0}$ denote category weight.

\begin{prop}
\label{PropCwgtSusp}
Let $u \in H^k(X;\QQ)$ with $u \neq 0$, $k \geq 2$. If $f^*u=0$ for all continuous $f:\Sigma P \to X$, where $P$ is an arbitrary closed oriented $(k-1)$-dimensional manifold, then $\cwgt(u)\geq 2$. 
\end{prop}

\begin{cor}
\label{CorCwgtAspher}
Let $\sigma \in H^ 2(X;\QQ)$ with $\sigma \neq 0$. If $\sigma$ is \emph{aspherical}, i.e. if $f^*\sigma=0$ for every continuous map $f:S^2 \to X$, then $\cwgt(\sigma)=2$. 
\end{cor}

\subsection{Spherical complexity weights}
Let $n \in \NN$. Similar to the situation for topological complexity, we want to focus on a particular type of cohomology classes of $S_nX$.

Let $E_n:S_nX \times S^n \to X$, $(\gamma,p) \mapsto \gamma(p)$, be the evaluation map and let $\sigma_{S^n} \in C_n(S^n)$ be a fundamental cycle. For $k >n$ we consider $\bar{Z}_n: C^k(X;A) \to C^{k-n}(S_{n}X;A)$, $\bar{Z}_n(\varphi) = E_{n}^*\varphi/\sigma_{S^{n}}$, and let
$$Z_n: H^k(X;A) \to H^{k-n}(S_nX;A), \qquad Z_n(u)=E_n^*u/[S^{n}],$$ be its induced map in cohomology.

\begin{lemma} 
\label{LemmaZn}
Let $\K$ be a field and $u \in H^k(X;\K)$ with $k > n$. Then $c_n^*(Z_n(u))=0$, where $c_n:X \to S_nX$ is the inclusion of constant maps.
\end{lemma}

\begin{proof}
Let $\varphi \in C^k(X;\K)$ be a representative of $u$. We want to evaluate $c_n^*\bar{Z}_n(\varphi)$ on cycles in $C_{k-n}(X;\K)$. If $\rho \in C_{k-n}(X;\K)$ is such a cycle, then 
$$\left<c_n^*\bar{Z}_n(\varphi),\rho\right>=\left<E_{n}^*\varphi/\sigma_{S^{n}},(c_n)_*(\rho) \right>= \left<\varphi,(E_n \circ (c_n \times \id_{S^n}))_*(\rho \otimes \sigma_{S^n})\right>.$$
But by definition of the maps, $E_n \circ (c_n \times \id_{S^n})= \pr_1:X\times S^n \to X$, where $\pr_1$ denotes the projection onto the second factor. Hence, $\left<\bar{Z}_n(\varphi),(c_n)_*(\rho) \right>=\left<\varphi,(\pr_1)_*(\rho \otimes \sigma_{S^n}) \right>$.
But since $(\pr_1)_*(\rho \otimes \sigma_{S^n})$ is a degenerate $k$-chain, passing to (co)homology yields
$$\left<c_n^*Z_n(u),[\rho] \right>= 0.$$
Thus, the non-degeneracy of the Kronecker pairing shows the claim.
\end{proof}

Note that Lemma \ref{LemmaZn} expresses that $\wgt_n(Z_n(u))\geq 1$ whenever $Z_n(u)\neq 0$. 

\begin{prop}
\label{Propwgtn2}
Let $\K$ be a field and $u \in H^k(X;\K)$ with $Z_n(u) \neq 0 \in H^{k-n}(S_nX;\K)$, where $k \geq n+2$. Let $\ev_{n+1}: \C^0(S^{n+1},X) \times S^{n+1} \to X$, $\ev_{n+1}(\gamma,p)=\gamma(p)$, be the evaluation map. If $$\ev_{n+1}^*u/[S^{n+1}]=0 \in H^{k-n-1}(\C^0(S^{n+1},X);\K),$$ then $\wgt_n(Z_n(u))\geq 2$.
\end{prop}

\begin{proof}
By Corollary \ref{CorGenwgt2} and Proposition \ref{PropMVSC}, it suffices to show that $F_n^*Z_n(u)/[S^1]=0$ under the given assumptions. For $\sigma \in H_{k-n-1}(\C^0(S^{n+1},X);\K)$ we compute that 
$$\left<F_n^*Z_n(u)/[S^1],\sigma \right>= \left<u,(G_n)_*(\sigma \times [S^1] \times [S^n]) \right>,  $$
where $G_n: \C^0(S^{n+1},X) \times S^1 \times S^n \to X$, $G_n(\gamma,t,p)=E_n(F_n(\gamma,t),p)$. Again let $1_n= (1,0,\dots,0) \in S^{n+1}$. By construction of $F_n$, it holds that $F_n(\gamma,0) = c_n(\gamma(1_n))$ for all $\gamma \in \C^0(S^{n+1},X)$. Thus, it follows that
\begin{equation}
\label{EqGn1}
G_n(\gamma,0,p) = E_n(F_n(\gamma,0),p)=\gamma(1_n) \qquad \forall \gamma \in \C^0(S^{n+1},X), \ p \in S^n.
\end{equation}
Moreover, it holds by construction of $F_n$  that $(F_n(\gamma,t))(1_n)=\gamma(1_n)$ for all $\gamma \in \C^0(S^{n+1},X)$, $t \in S^1$, and consequently,
\begin{equation}
\label{EqGn2}
G_n(\gamma,t,1_n)= \gamma(1_n) \qquad \forall \gamma \in \C^0(S^{n+1},X), \ t \in S^1.
\end{equation}
Equations \eqref{EqGn1} and \eqref{EqGn2} imply that $G_n$ induces a map $\tilde{G}_n:\C^0(S^{n+1},X) \times (S^1 \wedge S^n) \to X$. One checks without difficulties that there is a homeomorphism $\varphi: S^1 \wedge S^n \to S^{n+1}$, for which  $$\tilde{G}_n = \ev_{n+1} \circ (\id_{\C^0(S^{n+1},X)} \times \varphi).$$ If $q:S^1 \times S^n \to S^{n+1}$ denotes the composition of the projection $S^1 \times S^n\to S^1 \wedge S^n$ with $\varphi$, one observes that $q_*([S^1] \times [S^n]) = \pm [S^{n+1}] \in H_{n+1}(S^{n+1};\K)$. Inserting these considerations into the above computations yields 
$$\left<F_n^*Z_n(u)/[S^1],\sigma \right>= \pm \left<\ev_{n+1}^*u/[S^{n+1}],\sigma \right>. $$
Hence, $F_n^*Z_n(u)= \pm\ev_{n+1}^*u/[S^{n+1}]$, which shows the claim. 
\end{proof}

\begin{cor}
\label{CorZnwgt}
Let $u \in H^k(X;\QQ)$ with $Z_n(u) \neq 0 \in H^{k-n}(S_nX;\QQ)$, where $k \geq n+2$. If $f^*u=0$ for every continuous $f:P \times  S^{n+1} \to X$, where $P$ is any closed oriented $(k-n-1)$-dimensional manifold, then $\wgt_n(Z_n(u))\geq 2$.
\end{cor}
\begin{proof}
To show that $\ev_{n+1}^*u/[S^{n+1}]$ vanishes and apply Proposition \ref{Propwgtn2}, it suffices to consider its Kronecker pairings with homology classes of the form $f_*[P]$, where $f:P \to \C^0(S^{n+1},X)$ is continuous and $P$ is a closed oriented $(k-n-1)$-manifold. We compute that
$$\left<\ev_{n+1}^*u/[S^{n+1}],f_*[P]\right>=\left<g^*u,[P \times S^{n+1}]\right>, $$
where $g:P \times S^{n+1} \to X$, $g(x,p)=\ev_{n+1}(f(x),p)$. Hence, by assumption on $u$, it follows that $\left<\ev_{n+1}^*u/[S^{n+1}],f_*[P]\right>=0$ and since $P$ and $f$ were chosen arbitrarily, we derive that $\ev_{n+1}^*u/[S^{n+1}]=0$. The claim follows from Proposition \ref{Propwgtn2}.
\end{proof}

The case $k=n+2$ in the previous corollary takes the following concise form.

\begin{cor}
Let $u \in H^{n+2}(X;\QQ)$ with $Z_n(u) \neq 0 \in H^2(S_nX;\QQ)$. If $f^*u=0$ for every continuous $f:S^1 \times S^{n+1} \to X$, then $\wgt_n(Z_n(u))\geq 2$. 
\end{cor}




\section{Topological complexity and manifolds dominated by products}
\label{SectionTC}
Before applying the constructions from the previous sections to arguments involving $\SC_n$ for $n >0$, we want to derive new estimates for the topological complexity of certain closed manifolds. The first one is a straightforward consequence of the results from the previous section. All manifolds occuring in this section are assumed to be smooth and connected. 

Given any cohomology class $u \in H^*(X;R)$, where $X$ is a topological space and $R$ is a commutative ring, we again write $\bar{u}:= 1 \times u - u \times 1 \in H^*(X \times X;R)$.

\begin{theorem}
Let $M$ be a closed oriented manifold with $\dim M>2$. If $H^2(M;\QQ)$ contains an atoroidal class, then $\TC(M) \geq 6$. 
\end{theorem}
\begin{proof}
Let $n = \dim M$, let $u \in H^2(M;\QQ)$ be atoroidal. One computes that $\bar{u}^2\neq 0 \in H^4(M\times M;\QQ)$. Since $n>2$, there exists $v \in H^{n-2}(M;\QQ)$, such that $u\cdot v$ is a generator of $H^n(M;\QQ)$, where $\cdot$ denotes the cup product. One further computes that $\bar{u}^2 \cdot \bar{v}\neq 0$, since it contains the non-vanishing summand $-2u \times uv$. Theorem \ref{Thmwgtcup} and Corollary \ref{CorAtor} yield
$$\wgtTC\left(\bar{u}^2 \cdot \bar{v} \right) \geq 2 \wgtTC(\bar{u}) + \wgtTC(\bar{v}) \geq 5.$$
Hence, $\TC(M)\geq 6$ by part (1) of Theorem \ref{Thmwgtprops}.
\end{proof}

Let $M$ and $N$ be closed, oriented manifolds with $\dim M =\dim N$. We recall that $M$ \emph{dominates $N$}, and write $M \geq N$, if there exists a smooth map $M \to N$ of non-zero degree. See \cite{delaHarpe} for an overview of constructions and results involving dominated manifolds.

\begin{prop}
\label{PropDom}
Let $M$ be a closed oriented manifold of positive dimension. If $\TC(M) \leq 3$, then $M$ will be a rational homology sphere or it will be dominated by a manifold of the form $P \times S^1$, where $P$ is a closed oriented manifold of dimension $\dim M-1$.
\end{prop}
\begin{proof}
Let $n=\dim M$ and $\sigma \in H^n(M;\QQ)$ with $\sigma \neq 0$. Assume that $M$ was not dominated by any manifold of the above form. Then $f^*\sigma=0$ for all closed $(n-1)$-dimensional manifolds $P$ and continuous maps $f: S^1 \times P \to M$. Thus, $\wgtTC(\bar\sigma) \geq 2$ by Proposition \ref{Propsigmapull}. Assume also that $M$ is not a rational homology sphere, hence there exist $i \in \{1,2, \dots,n-1\}$ and $\nu \in H^i(M;\QQ)$ with $\nu \neq 0$. Then $\bar\sigma \cdot \bar\nu \neq 0$, since it contains the non-vanishing summand $\pm \nu \times \sigma$. Hence, $\wgtTC(\bar\sigma \cdot \bar\nu) \geq \wgtTC(\bar\sigma) + \wgtTC(\bar\sigma)=3$, which implies that $\TC(M) \geq 4$. 
\end{proof}

In the even-dimensional case, one can make the following stronger statement. 

\begin{prop}
\label{PropDomEven}
Let $M$ be an even-dimensional closed oriented manifold. If $\TC(M) \leq 4$, then $M$ will be dominated by a manifold of the form $P \times S^1$, where $P$ is a closed oriented manifold of dimension $\dim M-1$.
\end{prop}
\begin{proof}
Let $n=\dim M$ and $\sigma \in H^n(M;\QQ)$ with $\sigma \neq 0$ and assume that $M$ was not dominated by a manifold of the form $P \times S^1$. Then again $\wgt(\bar\sigma)\geq 2$. Since $\sigma$ is of even degree, one computes that $\bar\sigma^2 = -2 \sigma \times \sigma \neq 0$, hence $\wgt(\bar\sigma^2) \geq 4$, which implies that $\TC(M) \geq 5$. 
\end{proof}

In the case of highly connected manifolds, the domination condition from Proposition \ref{PropDomEven} suffices to determine the value of $\TC$.

\begin{cor}
\label{CorHighly}
 Let $n \in \NN$ with $n \geq 2$ and let $M$ be a closed, oriented, $(n-1)$-connected $2n$-dimensional manifold. If there is no closed oriented $(2n-1)$-dimensional manifold $P$, such that $M$ is dominated by $P \times S^1$, then $\TC(M)=5$.
\end{cor}
\begin{proof}
If $M$ is an $(n-1)$-connected $2n$-manifold, then the standard upper bound for $\TC$ by dimension and connectivity yields $\TC(M) \leq \frac{4n+1}{n}+1= 5+ \frac1n$, i.e. $\TC(M) \leq 5$, since $\TC(M)$ is integer-valued. The other inequality is an immediate consequence of Proposition \ref{PropDomEven}.
\end{proof}

Manifolds that are dominated by product manifolds have been studied by D. Kotschick and C. L\"oh in \cite{KotLoeh}. Combining our previous results with their studies yields a more tangible sufficient condition on a closed manifold for its topological complexity to be four or bigger. 

\begin{theorem}
Let $M$ be a closed oriented connected manifold which is rationally essential, i.e. for which $(c_M)_*([M])\neq 0\in H_*(B\pi_1(M);\QQ)$, where $c_M: M \to B\pi_1(M)$ denotes a classifying map of the universal covering of $M$. Assume that $\pi_1(M)$ is hyperbolic, but not virtually cyclic. 
\begin{enumerate}[a)]
\item If $M$ is not a rational homology sphere, then $\TC(M) \geq 4$.
\item If $M$ is even-dimensional, then $\TC(M) \geq 5$. 
\end{enumerate}
 \end{theorem}
\begin{proof}
This follows from combining Propositions \ref{PropDom} and \ref{PropDomEven} with Theorems 1.4 and 1.5 of \cite{KotLoeh}.
\end{proof}

We conclude this section by studying the topological complexity of closed three-dimensional manifolds (3-manifolds). The LS category of closed 3-manifolds has been fully determined by J. C. G\'{o}mez-Larra\~{n}aga and F. Gonz\'{a}lez-Acu\~{n}a in \cite{GLGA}, but to the author's best knowledge a systematic study of the topological complexity of closed 3-manifolds has not been conducted anywhere in the literature. 

\begin{theorem}
\label{Theorem3dim}
Let $M$ be a closed oriented three-dimensional manifold that is not a rational homology $3$-sphere. If $M$ is not dominated by any closed oriented product manifold, then $5 \leq \TC(M) \leq 7$.
\end{theorem}
\begin{proof}
By the dimensional upper bound for topological complexity, $\TC(M) \leq 7$ for any three-dimensional manifold. Let $c \in H^3(M;\QQ)$ be a generator. Any closed oriented three-dimensional product manifold has the form $\Sigma \times S^1$, where $\Sigma$ is a closed oriented surface. Thus, if $M$ is not dominated by a product manifold, then $f^*u=0$ for all continuous $f: \Sigma \times S^1 \to M$, where $\Sigma$ is any closed oriented surface. Proposition \ref{Propsigmapull} then implies that $\wgtTC(\bar{c}) \geq 2$.

Let $\sigma_1 \in H^1(M;\QQ)$ and $\sigma_2 \in H^2(M;\QQ)$ such that $\sigma_1 \cdot \sigma_2 =c$. Since $M$ is by assumption not a rational homology sphere, such classes exist by Poincar\'e duality. A straightforward computation shows that $\bar{c} \cdot \bar\sigma_1 \cdot \bar\sigma_2=-2 c \times c \neq 0$. By Theorem \ref{Thmwgtcup}, this yields 
$$\wgtTC(\bar{c} \cdot \bar\sigma_1 \cdot \bar\sigma_2) \geq \wgtTC(\bar{c})+\wgtTC(\bar\sigma_1)+\wgtTC(\bar\sigma_2) \geq 4,$$
so that $\TC(M) \geq 5$. 
\end{proof}

\begin{remark}
In \cite[Theorem 1]{KotNeo}, Kotschick and C. Neofytidis have worked out geometric conditions on a closed oriented 3-manifold $M$ for being dominated by a product manifold. Thus, the assumptions of Theorem \ref{Theorem3dim} might be replaced by more explicit ones.

Moreover, they show that the total space of a circle bundle over a closed oriented surface of positive genus with non-zero Euler number satisfies these conditions, see \cite[Lemma 1]{KotNeo}.
\end{remark}

\section{$\SC_1$ and closed geodesics}
\label{SectionGeod}

As a geometric application of spherical complexities we want to apply the results of the previous sections to the case $n=1$ and to energy functionals of Riemann/Finsler manifolds as outlined in subsection \ref{SubFinsler} to derive new existence results for closed geodesics. As a starting point for a reader who wants to learn about closed geodesics, the author recommends the lucid survey \cite{OanceaGeod} summarizing many important results for the Riemannian case.

We will first outline a general strategy of employing classes of spherical complexity weight two in order to show the existence of two geometrically (positively) distinct closed geodesics in certain cases. Again we assume that all manifolds occuring in this section are smooth and connected. 


\begin{lemma}
\label{LemmaJones}
Let $R$ be a commutative ring. If $X$ is simply connected, then $$Z_1:H^k(X;R) \to H^{k-1}(LX;R)$$ will be injective for $k \geq 1$.
\end{lemma}

\begin{proof}
We consider $C^{-*}(X;R)$, i.e. $C^*(X;R)$ with inverted grading, equipped with the singular differential and the cup product, as a differential graded algebra. In \cite{Jones}, following ideas of Anderson and Bott-Segal (see \cite{BottSegal}), J.D.S. Jones constructed a chain map 
$$\Phi: CH_*(C^{-*}(X;R)) \to C^*(LX;R), $$
where again $LX=\C^0(S^1,X)$ and where $CH_*(C^{-*}(X;R))$ is the Hochschild chain complex of $C^{-*}(X;R)$, that induces an isomorphism  $\Phi_*: HH_{-*}(C^{-*}(X;R)) \to H^*(LX;R)=H^*(S_1X;R)$ in (co)homology if $X$ is simply connected, see also \cite{LodayFree} or \cite[Section 3]{CohenJones} for descriptions of $\Phi$.

Let $u \in H^k(X;R)$ with $u \neq 0$ and $k\geq 1$ and let $x \in C^k(X;R)$ be a representative of $u$. We consider $1 \otimes x \in C^0(X;R)\otimes_R C^k(X;R) \subset CH_*(C^{-*}(X;R))$. By definition of Hochschild homology, one checks   that $1 \otimes x$ is a Hochschild cycle whose homology class is non-vanishing. Studying Jones' construction, we observe that by definition of $\Phi$ it holds that
$$\Phi_*([1 \otimes x]) = \ev_1^*u/[S^1] = Z_1(u).$$
Hence, $Z_1(u)\neq 0$ if $X$ is simply connected, which shows the claim.
\end{proof}

The following statement combines results from the previous sections in the case $n=1$. Given a Riemannian manifold $M$ with energy functional $E: \Lambda M \to \RR$ and $x \in H^*(\Lambda M;R)$ we let
$$\ccr(x):= \sup \{\lambda \in \RR \; | \; x \in \ker [i_\lambda^*:H^*(\Lambda M;R) \to H^*(E^{<\lambda};R)] \}.$$

\begin{prop}
\label{PropSClambda3}
Let $M$ be a simply connected $n$-dimensional closed oriented Riemannian manifold, where $n \geq 3$. Let $k \geq 3$ and $u \in H^k(M;\K)$ with $u \neq 0$, where $\K$ is a field. If $f^*u=0$ for every continuous map $f:P \times S^2 \to M$, where $P$ is any closed oriented $(k-2)$-dimensional manifold, then $\SC_{1,M}(E^\lambda) \geq 3$ for every $\lambda \in \RR$ with $\lambda \geq \ccr(Z_1(u))$. 
\end{prop}

\begin{proof}
Since $M$ is simply connected, it follows from Lemma \ref{LemmaJones} that $Z_1(u) \neq 0$. By assumption on $u$,  and Corollary \ref{CorZnwgt}, it follows that $\wgt_1(Z_1(u))\geq 2$. For $\lambda \in \RR$ we let $i_\lambda: E^\lambda \hookrightarrow \Lambda M$ be the inclusion of the corresponding sublevel set. If $\lambda \geq \ccr(Z_1(u))$, then $i_\lambda^*Z_1(u) \neq 0 \in H^*(E^\lambda;\K)$. Thus, it follows from Theorem \ref{ThmScnEnrel} that $\SC_{1,M}(E^\lambda) \geq 3$. 
\end{proof}

It follows from Theorem \ref{TheoremLSFinsler} that under the assumptions of Proposition \ref{PropSClambda3}, there will be two $O(2)$-orbits of non-constant closed geodesics in $E^\lambda$. A priori, both  might be orbits of iterates of the same prime closed geodesic, but some additional arguments can help in excluding this case. 

\begin{strategy}
Let $M$ be a simply connected $n$-dimensional closed oriented Finsler manifold with energy functional $E:\Lambda M \to \RR$, where $n \geq 3$. 
\begin{enumerate}
\item Find $u \in H^k(M;\QQ)$, where $k \geq 3$, such that $f^*u=0$ for all continuous $f:P \times S^2 \to M$, where $P$ is an arbitrary $(n-2)$-dimensional closed oriented manifold. 
\item Find $\lambda \in \RR$ with the following properties:
\begin{enumerate}[(i)]
\item $ \lambda \geq \ccr(Z_1(u))$,
\item there exists a  non-constant prime closed geodesic with $E(\gamma) \leq \lambda$, but $E(\gamma^m)> \lambda$ for all $m >1$.
\end{enumerate}
\end{enumerate}
If such $u$ and $\lambda$ have been found, then Proposition \ref{PropSClambda3} and Theorem \ref{TheoremLSFinsler} will imply that $E^\lambda$ contains two $O(2)$-orbits of closed geodesics in the reversible case and $SO(2)$-orbits of closed geodesics in the non-reversible case. By property (ii) of $\lambda$ one of them can be assumed to be a prime closed geodesic whose iterates all lie outside of $E^\lambda$, hence the second orbit  has to contain a closed geodesic geometrically distinct from the first one in the reversible case and positively distinct in the non-reversible case. 
\end{strategy}

The upcoming Theorem \ref{TheoremGeodPosCurv} provides an example of this strategy being employed. We will require an assumption on the structure of the cohomology ring of $M$ and a pinching condition on the flag curvature of the Finsler metric. The author recommends H.-B. Rademacher's article \cite{RadeNonrev} as a general reference on Finsler manifolds of positive flag curvature.

\begin{definition}[{\cite{RadeNonrev},\cite{RadeSphere}}]
Let $M$ be a closed manifold and let $F:TM \to \RR$ be a Finsler metric on $M$. The \emph{reversibility of $F$} is the number 
$$\lambda := \sup \{F(x,-v) \ | \ (x,v) \in TM,\ F(x,v)=1\} \geq 1.$$
Note that $F$ is reversible if and only if $\lambda =1$.
\end{definition}

The Riemannian case of the following theorem follows from the stronger statement \cite[Theorem D]{BTZExistence} of Ballmann, Thorbergsson and Ziller. Moreover, Duan, Long and Wang have shown in \cite{DuanLongWang} that given a compact and simply connected manifold, any Finsler metric that is non-reversible and bumpy, i.e. its energy is a Morse-Bott function and has isolated critical $SO(2)$-orbits, has two positively distinct closed geodesics. 

An existence result of more than two positively distinct closed geodesics for Finsler metrics on spheres under a pinching condition stricter than ours has been shown by Wang in \cite{WangMultiple}.

\begin{theorem}
\label{TheoremGeodPosCurv}
Let $n \geq 3$ and let $M$ be a $2n$-dimensional closed oriented manifold. Assume that there exists a cohomology class $x \in H^{2k}(M;\QQ)$ with $x^2 \neq 0$, where $1 \leq k \leq \frac{n-1}2$. Let $F$ be a Finsler metric on $M$ of reversibility $\lambda$. If the flag curvature $K$ of $F$ satisfies 
$$\frac{1}{4}\Big(\frac{\lambda}{1+\lambda}\Big)^2 < \delta \leq K \leq 1,$$ 
then $(M,F)$ will have two positively distinct closed geodesics of length at most $\frac{\pi}{\sqrt{\delta}}$. 

If in addition $F$ is reversible, e.g. if $F$ is derived from a Riemannian metric whose sectional curvature satisfies $\frac{1}{16} <\delta \leq K \leq 1$, then $(M,F)$ will have two geometrically distinct closed geodesics of length at most $\frac{\pi}{\sqrt{\delta}}$.
\end{theorem}

\begin{proof}
 Since $M$ is even-dimensional, oriented and has a metric of positive sectional curvature, $M$ is simply connected by Synge's theorem, which was generalized to Finsler manifolds by L. Auslander in \cite{Auslander}. This in particular	 implies that $u :=Z_1(x^2)\in H^{4k-1}(\Lambda M;\QQ)$ satisfies $u \neq 0$ by Lemma \ref{LemmaJones}. To show that $\wgt_1(u) \geq 2$, it suffices by Corollary \ref{CorZnwgt} to show that $f^*(x^2)=0$ for each continuous map $f:S^2 \times P \to X$, where $P$ is an arbitrary closed oriented $(4k-2)$-dimensional manifold. But since $H^*(S^2 \times P;\QQ) \cong H^*(S^2;\QQ) \otimes H^*(P;\QQ)$ are isomorphic as rings, $H^*(S^2 \times P;\QQ)$ contains no non-trivial squares in degree $4k$, which implies that $f^*(x^2)= (f^*x)^2=0$ for each such $f$. Hence, $\wgt_1(u)\geq 2$. 

Let $E: \Lambda M \to \RR$ be the energy functional and $L: \Lambda M \to \RR$ be the length functional of $F$. Since $u \neq 0$, there exists a non-constant closed geodesic $\gamma_0$ with $E(\gamma_0)=\ccr(u)$ and 
$$\ind(\gamma_0) \leq \deg u=4k-1.$$ 
(See e.g. \cite{RadeHabil} or \cite{KlinGeod} for details.) To apply the strategy outlined above, we want to derive estimates on $\ccr(u)$ from the given curvature bounds. 

Let $\delta> \frac{1}{4}(\frac{\lambda}{\lambda+1})^2$ be such that the flag curvature of $F$ satisfies $K \geq \delta$. Since $2k \leq n-1$ by assumption, it holds that $\ind(\gamma_0) < 2n-1$. Thus, we can apply \cite[Lemma 3]{RadeSphere} (see also \cite[Remark 8.5]{RadeNonrev}), which generalizes \cite[(1.8)]{BTZclosed} from Riemannian to Finsler metrics, and obtain that $L(\gamma_0) \leq \pi/\sqrt{\delta}$ and thus $E(\gamma_0)=(L(\gamma_0))^2 \leq \pi^2/\delta$ which yields
\begin{equation}
\label{Eqccrup}
\ccr(u) \leq \frac{\pi^2}{\delta} =:a. 
\end{equation}
Let $\iota_a: E^{a} \hookrightarrow \Lambda M$ denote the inclusion of the closed sublevel set. Put $G= SO(2)$ if $F$ is non-reversible and $G=O(2)$ if $F$ is reversible. By the previous computation it holds that $\iota_a^*u \neq 0 \in H^{2k-1}(E^{a};\QQ)$, so it follows from Theorem \ref{ThmScnEnrel} that $\SC_{1,M}(E^a) \geq 3$. By Theorem \ref{TheoremLSFinsler}, this implies that $E^a$ contains two $G$-orbits of closed geodesics. A priori, these two $G$-orbits of closed geodesics might be orbits of iterates of the same prime closed geodesic. We want to exclude this case by showing that $E^a$ contains no iterated closed geodesics under the given flag curvature assumptions.
 
Since the flag curvature of $F$ is positive and bounded from above by one, it follows from \cite[Theorem 10.3]{RadeNonrev} that under the given upper bound on the flag curvature of $F$ every non-constant closed geodesic of $F$ satisfies $L(\gamma) \geq \pi(1+\lambda^{-1})$, hence
\begin{equation}
\label{EqccFlow}
E(\gamma) \geq \pi^2\Big(1+ \lambda^{-1}\Big)^2=\pi^2 \Big(\frac{1+\lambda}\lambda \Big)^2.
\end{equation}
(In the Riemannian case, this is a direct consequence of Klingenberg's injectivity radius estimate, see \cite[Theorem 6.5.1]{Petersen}.) Its iterates thus satisfy $$E(\gamma^m) \geq E(\gamma^2) \geq 4\pi^2\Big(\frac{1+\lambda}\lambda\Big)^2$$ for all $m \geq 2$. But by assumption on $\delta$, 
$$a = \frac{\pi^2}{\delta} < 4 \pi^2\Big(\frac{1+\lambda}\lambda\Big)^2.$$
Comparing the previous two inequalities, we observe that $E^a$ cannot contain any iterated closed geodesic of $F$, hence every closed geodesic contained in $E^a$ is prime. This shows that $E^a$ contains two positively distinct prime closed geodesics if $F$ is non-reversible and two geometrically distinct prime closed geodesics if $F$ is reversible. Moreover, the length of each of these geodesics satisfies $L(\gamma) \leq \sqrt{E(\gamma)} \leq \sqrt{a}=\pi/\sqrt{\delta}$.
\end{proof}

\begin{remark}
\begin{enumerate}
\item The topological assumption of Theorem \ref{TheoremGeodPosCurv} is satisfied by all closed oriented manifolds having the rational homotopy type of a simply connected closed symplectic manifold of dimension six or bigger whose symplectic form represents a rational cohomology class. In this case, the cohomology class $x$ in the statement of the theorem may be taken as that very class in the case $k=1$. 
\item That assumption is further satisfied by any closed oriented manifold $M$ with $H^*(M;\QQ) \cong \QQ[x]/(x^{d+1}),$ where $d \geq 3$ and $x$ has even degree. This in particular applies to manifolds having the rational homotopy type of $\CP^n$ or $\mathbb{H}P^n$ with $n \geq 3$. 
\end{enumerate}
\end{remark}

 \bibliography{TCHH}
 \bibliographystyle{amsalpha}

\end{document}